%% file: 00_main_arxiv.tex
\documentclass[a4paper,11pt,reqno]{amsart}

\input{00_macros}

\begin{document}

\title[Exact Mixed-Integer Conic Liftings for Queueing-Based CDN Design]{\Large Exact Mixed-Integer Conic Liftings for Queueing-Based Content Delivery Network Design}

\author[V. Blanco and M. Mart\'inez-Ant\'on]{
{\large V\'ictor Blanco$^{\dagger}$ and Miguel Mart\'inez-Ant\'on$^{\ddagger}$}\bigskip\\
\small $^\dagger$Institute of Mathematics (IMAG), Universidad de Granada\\
\small $^\ddagger$Institute of Mathematics (IMUS), Universidad de Sevilla\bigskip\\
\small \texttt{vblanco@ugr.es}, \texttt{mmartinez31@us.es}
}

\maketitle

\begin{abstract}
We study a class of content delivery network design problems in which service performance is explicitly modeled through queueing-based response times. In contrast to standard formulations relying on distance-based approximations, the resulting models incorporate congestion effects and give rise to nonlinear expressions involving ratios of affine mappings, reciprocal terms, and interactions across multiple traffic classes. To address these challenges, we develop a systematic framework based on mixed-integer conic liftings that enables exact reformulations of such expressions within a tractable optimization paradigm. The proposed approach combines McCormick-type envelopes and second-order cone representations to derive mixed-integer conic programming formulations that jointly capture geometric design decisions and stochastic service dynamics within the same system. The framework is first introduced in a single-server setting, where the optimal solution exhibits a connection with the classical Fermat--Weber problem, and is then extended to the multi-server case, leading to a general formulation that simultaneously optimizes server locations, capacity allocation, and routing decisions. Computational experiments on a realistic case study illustrate the impact of congestion-aware modeling on network design and demonstrate the effectiveness of the proposed formulations when solved using modern conic optimization solvers.
\end{abstract}
 

\section{Introduction}

Mixed-integer conic optimization provides a unified framework for modeling problems that combine discrete decisions with nonlinear convex structures \citep{BenTalNemirovski2001,BoydVandenberghe2004}. A central aspect of this framework is the construction of extended formulations in which nonlinear expressions are lifted into higher-dimensional spaces and represented through conic constraints \citep{conforti2010extended,vielma2015mixed}. Such representations are particularly effective when the underlying models involve ratios, reciprocal terms, or interactions between multiple variables, since many of these functions admit exact or tight conic representations through standard cones. More generally, a broad class of convex optimization problems can be expressed, or arbitrarily well approximated, using combinations of a small family of canonical cones, often referred to as the \emph{Big Five}, namely the nonnegative orthant, second-order cone, positive semidefinite cone, power cone, and exponential cone \citep{bertsimas2026cone}. This expressive power explains the broad applicability of conic optimization as a modeling paradigm.

Nevertheless, a fundamental challenge in this area is to identify classes of optimization problems whose intrinsic nonlinear structure can be systematically represented through exact conic liftings while preserving low representational complexity. Although many standard convex sets and functions admit well-known conic formulations, more intricate nonlinear expressions arising in stochastic systems, congestion models, and service-network optimization are often handled through {\it ad hoc} reformulations or piecewise-linear approximations \citep{pham2025exponential,atamturk2010conic}. In particular, queueing-based optimization models naturally generate rational expressions involving reciprocal load factors, ratios of affine mappings, and nonlinear interactions between traffic and service variables, whose exact representability is far from immediate. Developing optimization frameworks in which conic lifting is not merely a computational device, but rather the natural representational language of the underlying stochastic structure, remains an important direction of research in mixed-integer conic optimization.

This paper develops such a framework in the context of the design of content delivery networks (CDNs for short), which constitutes a fundamental component of modern Internet infrastructure~\citep{pallis2006insight,verma2003cdn}. CDNs aim to deliver digital content, such as web pages, video streams, software updates, and cloud services, to geographically distributed users with low latency and high reliability. To achieve this objective, content is replicated across a set of distributed surrogate servers, referred to as \emph{edge servers}, which are deployed closer to users than the original repositories, known as \emph{origin servers}. When a user issues a request, it is routed to an appropriate edge server capable of serving the content locally whenever possible, while cache misses require retrieving the content from the origin layer.

This architecture significantly reduces end-to-end latency, alleviates network congestion, and improves scalability~\citep{liu2012optimizing}. As a consequence, CDNs play a central role in modern digital services, including large-scale video streaming platforms~\citep{ghabashneh2020exploring,krishnappa_optimizing_2015}, cloud computing infrastructures~\citep{ling2013cdn}, software distribution systems~\citep{nygren2010akamai}, and emerging edge computing environments supporting latency-sensitive applications such as online gaming, augmented reality, and Internet-of-Things services~\citep{aral2019addressing,george2021evolution}.

From an optimization perspective, the efficient design and operation of CDNs raises several challenging problems involving the placement of edge servers, the allocation of user demand, the routing of requests across the network, the management of content replication, and the provisioning of service capacity. These decisions strongly affect both system performance and operational costs. A substantial body of research has examined these questions from both optimization and performance viewpoints \citep{verma2003cdn,buyya_content_2008,pallis2006insight,nygren2010akamai}. Early works addressed pricing and economic aspects of CDN deployment \citep{hosanagar_service_2008}, while subsequent contributions formulated server placement and mirror selection problems using classical facility location and $k$-center models \citep{hillmann_modeling_2016}. 

Further developments have considered the joint optimization of caching and routing decisions~\citep{poularakis_video_2014,guo_joint_2019,ghabashneh2020exploring}, as well as network-level optimization problems such as transit routing and traffic engineering~\citep{liu2012optimizing,ahmed_optimizing_2018}. More recently, cloud-oriented CDN architectures have been studied, where resource allocation and replica placement are integrated within virtualized environments~\citep{papagianni_cloud-oriented_2013,ling2013cdn}. Comprehensive surveys of these approaches can be found in \cite{salahuddin_survey_2018,pathan_taxonomy_2007}.

Despite this extensive literature, most existing optimization models and methods rely on simplified representations of service performance, typically approximating response times through deterministic proxies based on geographic distance or clustering strategies~\citep{verma2003cdn,pallis2006insight}. However, these approximations neglect congestion effects arising from the stochastic dynamics of request processing at edge servers. In practice, user-perceived latency depends not only on propagation delays, but also on queueing effects induced by heterogeneous service requirements and fluctuating traffic conditions, as highlighted in optimization-based overlay CDN models that explicitly account for dynamic traffic and resource allocation under real-time constraints \citep{caviglione2011overlay}. Incorporating these effects naturally leads to nonlinear expressions derived from queueing theory involving ratios of affine mappings, reciprocal terms, and interactions between multiple flow components, as observed in standard queueing models \citep{gross2008fundamentals,kleinrock1975queueing}.

These queueing-induced expressions define convex performance measures over the stability region of the system, but do not admit direct formulations compatible with standard mixed-integer optimization methodologies. Rather than approximating or linearizing such nonlinearities, in this paper we show that they possess an intrinsic structure particularly amenable to exact conic lifting. More precisely, by identifying a family of primitive functions capturing the essential building blocks of the queueing-performance expressions, we derive structured extended formulations based on second-order and rotated quadratic cones. This reveals that conic reformulations are not merely auxiliary computational tools, but rather the natural representational framework for this class of congestion-aware stochastic optimization problems.

Within this framework, we study the design of CDN architectures that explicitly account for congestion effects at edge servers. We consider a network composed of geographically distributed demand points, a set of origin servers, and a collection of edge servers whose positions and service capacities must be determined. Each demand point generates requests according to a stochastic process, and each request may either be served locally at the edge layer, cache hit, or require interaction with the origin layer, cache miss, leading to heterogeneous service requirements and congestion dynamics. The resulting performance measures combine deterministic propagation delays with queueing-based stochastic congestion terms.

The corresponding optimization problem integrates continuous location decisions, binary assignment variables, and nonlinear queueing-performance expressions within a unified mixed-integer conic framework. The nonlinear terms arising from the queueing systems are embedded into the optimization problem through exact conic liftings that preserve the stochastic structure of the original models while remaining computationally tractable. Beyond CDN design, the proposed methodology applies more broadly to service-system design, telecommunications networks, transportation systems, and cloud-computing infrastructures where congestion effects must be integrated with discrete design decisions.

We summarize below the main theoretical contributions by means of the relation of results that the reader will find throughout this paper.

\begin{itemize}

\item[--] 
Theorems \ref{lemma:dsr} and \ref{lemma:isr} characterize the queueing-theoretic dynamics of a CDN under disaggregated and integrated service regimes, respectively, deriving explicit stochastic representations for the corresponding congestion systems.

\item[--] Theorem \ref{thm:cdn_service_regimes} establishes tractable explicit expressions for the expectation of the stochastic total response time associated with requests generated at demand points within the CDN architecture.

\item[--] 
Theorems \ref{th:expectation_dsr} and \ref{th:expectation_isr} derive exact conic liftings for the expected sojourn times arising under the disaggregated and integrated queueing regimes, respectively, proving the conic representability of the associated congestion expressions in all considered scenarios.

\item[--] Theorem \ref{th:general_cdn} develops an exact mixed-integer conic formulation for the general congestion-aware CDN design problem by systematically combining the primitive conic representations derived throughout the paper.

\item[--] Theorem \ref{th:complexity} establishes the NP-hardness of the resulting congestion-aware mixed-integer conic optimization framework.

\item[--] Theorem \ref{th:single_edge_continuous} proves the exactness of the continuous relaxation in the congestion-aware CDN design for a single edge server.
\end{itemize}

From a practical perspective, this paper also provides several contributions related to the modeling, optimization, and evaluation of modern congestion-aware CDN infrastructures.

\begin{itemize}

\item[--] We develop a unified optimization framework for the design of content delivery networks under different queueing dynamics, including uncongested ({\rm UNC}), disaggregated ({\rm DSR}), and integrated ({\rm ISR}) service regimes.
\item[--] We propose tractable mixed-integer conic optimization formulations that can be solved using off-the-shelf conic solvers, such as \mbox{\rm MOSEK}, thereby enabling the exact optimization of congestion-aware CDN architectures within a modern conic optimization framework.
\item[--] For single-server problems, we derive a continuous conic optimization formulation that remains valid in the presence of multiple origin servers, showing that once the discrete edge-location structure is removed, the resulting congestion-aware capacity and origin-assignment problem can be represented exactly within a tractable conic framework.
\item[--] We provide a flexible optimization methodology for analyzing the interaction between congestion effects, infrastructure deployment, service capacities, routing decisions, and robustness-oriented performance criteria in large-scale CDN systems.
\item[--] We conduct an extensive computational study based on benchmark instances generated from a real Internet topology dataset~\citep{ITDK}, assessing the scalability, practical behavior, and managerial implications of the proposed congestion-aware mixed-integer conic formulations under realistic Internet-topology conditions.
\end{itemize}

The remainder of the paper is organized as follows.
Section~\ref{sec:MBC} introduces the conic lifting framework underlying the proposed methodology.
Section~\ref{sec:prelims} presents the CDN design problem together with the associated queueing-based performance model.
Sections~\ref{sec:singleServer} and~\ref{sec:conicOPT} are devoted to the analysis of the single-server and multi-server settings, respectively, and to the derivation of the corresponding optimization formulations.
Section~\ref{sec:computational} reports computational results on a case study, and Section~\ref{sec:conclusion} concludes the paper.


\section{Conic Liftings and Representability}\label{sec:MBC}

The modeling framework developed in this paper relies on the systematic use of conic liftings for nonlinear expressions involving both continuous and binary variables. These expressions arise naturally in the queueing-performance metrics associated with the considered congestion-aware service systems and involve ratios, products, reciprocal terms, and nonlinear interactions between traffic and service-capacity variables that cannot be directly handled within standard mixed-integer linear or convex formulations.

Rather than treating these nonlinearities individually, we adopt a representability-oriented perspective based on the identification of primitive conic-representable functions and their composition rules. The objective is to derive exact structured liftings for the queueing-induced expressions appearing throughout the optimization framework while preserving low representational complexity.

To address these challenges, this section first recalls the notion of conic lifting and its role in reformulating nonlinear convex sets and functions through structured constraints. We then introduce the family of canonical cones used throughout the paper, including $p$-order and rotated quadratic cones, which are sufficient to model the nonlinear expressions arising in the considered queueing systems. A key ingredient in this construction is the treatment of products involving binary variables, for which McCormick envelopes and lifting arguments will be systematically employed. Finally, we show how these ingredients can be integrated into a unified mixed-integer conic framework coupling discrete design decisions with nonlinear stochastic performance measures.

From a methodological standpoint, conic optimization provides a unifying framework for representing nonlinear convex sets and functions through structured constraints. Many functions arising in applications, including quadratic-over-linear terms, norms, geometric means, and exponential expressions, admit exact representations through second-order, positive semidefinite, power, or exponential cones. This allows one to transform nonlinear convex optimization problems into structured conic programs solvable through interior-point methods (IPM) implemented in modern conic optimization solvers \citep{BenTalNemirovski2001,BoydVandenberghe2004}.

A key feature of conic reformulations is their compatibility with discrete optimization. When combined with binary decision variables, conic constraints yield mixed-integer conic programs that have become increasingly tractable due to major advances in modern optimization software. This is particularly relevant in settings where discrete design decisions must be integrated with nonlinear congestion metrics, as in the problems considered in this paper.

Formally, a convex set $\mathcal{X} \subseteq \mathbb{R}^n$ is said to be \emph{conic representable} if there exists a proper cone $\mathcal{K} \subseteq \mathbb{R}^d$, and an affine mapping
$$
(x,w)\mapsto
A
\begin{pmatrix}
x\\
w
\end{pmatrix}
-b
:
\mathbb{R}^{n+m}\to\mathbb{R}^d
$$
such that
$$
x \in \mathcal{X}
\iff
\exists w \in \mathbb{R}^m :
A
\begin{pmatrix}
x\\
w
\end{pmatrix}
-b
\in \mathcal{K}.
$$

In this case, the set $\mathcal{X}$ can be viewed as the projection onto the original variable space of a higher-dimensional feasible region defined through conic constraints. Any such lifted feasible region is called a \emph{conic lifting} of $\mathcal{X}$. Similarly, a convex function $f$ is said to be conic representable whenever its epigraph admits such a representation.

The importance of conic representability in the present setting stems from the fact that the queueing-performance expressions arising in congestion-aware CDN design generate convex epigraphs involving coupled reciprocal and rational structures. Although these nonlinearities are not directly compatible with standard mixed-integer optimization techniques, we show throughout the paper that they admit exact conic liftings through suitable auxiliary-variable constructions and composition arguments.

The conic liftings used throughout the paper rely on a small family of fundamental cones. In particular, we consider the nonnegative orthant $\mathbb{R}^m_+$ and the $p$-order cone
$$
\mathcal{L}^m_p :=
\{
(x,t)\in\mathbb{R}^m\times\mathbb{R}
:
\|x\|_p \le t
\},
$$
where $\|\cdot\|_p$ denotes the $\ell_p$-norm. When $p=2$, this cone corresponds to the second-order, or Lorentz, cone. In addition, we use the rotated quadratic cone
$$
\mathcal{S}^2_+ :=
\{
(x,y,z)\in\mathbb{R}^3
:
xy\ge z^2,\ x,y\ge0
\},
$$
which is equivalent to the cone of positive semidefinite $2\times2$ matrices.

These cones belong to the standard family of canonical cones and are sufficient to represent the nonlinear expressions considered throughout this work.

Moreover, whenever $p$ is rational, all these cones admit representations involving only a minimum number of second-order cone constraints, yielding structured and complexity-optimal second-order cone formulations~\citep{blanco2024minimal,blanco2025complexity}. This property will play a central role in the derivation of the exact mixed-integer conic formulations developed in the subsequent sections.

\section{Content Delivery Network Design Problem}\label{sec:prelims}

We are given a finite set of demand points $\I$. Each demand point $i\in\I$
is located at position $a_i\in\R^m$ and generates requests according to an independent
Poisson process with intensity $\xi_i>0$.
A fixed natural number $n\ge 1$ of edge servers is to be deployed, with continuous
locations $x_j\in\R^m$, $j\in\J = \{1, \ldots, n\}$.
A finite set of origin servers indexed by $\K$ is given, where each
origin server $k\in\K$ is located at $b_k\in\R^m$.
Origins are assumed to be fixed and sufficiently provisioned, hence they are not
decision variables in the design problem considered here.

The CDN design problem studied in this paper combines continuous location decisions,
discrete assignment decisions, and service-capacity decisions. Its distinctive feature is
that the quality of service is not approximated only through deterministic distances, but
is instead derived from a queueing representation of congestion at the edge layer. This
leads to nonlinear response-time expressions whose structure will later be exploited to
derive conic liftings.

Each edge server operates as a finite-capacity service system.
The service requirement of a request depends on cache availability.
If the requested object is stored at the edge, the request is a \emph{cache hit} and is
completed after local processing.
If the object is not available, the request is a \emph{cache miss}, and the edge must
retrieve the object from an assigned origin server before completing the
response.
Cache misses, therefore, entail additional network interaction and processing
overhead and lead to larger effective service requirements.

To represent propagation and access effects, we are given a metric $\D$ in $\R^m$.
If demand point $i\in\I$ is assigned to edge server $j\in\J$, the deterministic
access delay is modeled as $\kappa_1 \D(a_i,x_j)$, where $\kappa_1>0$ is a propagation delay per unit length.
In addition, when a cache miss occurs, edge server $j$ retrieves the content from
an assigned origin server.
The deterministic upstream retrieval delay on a miss is modeled as
$\kappa_2 \D(x_j,b_k)$, where $\kappa_2>0$ is the corresponding delay per unit length.

To capture processing and congestion effects at the edge servers, we adopt a general modeling framework in which each edge server $j\in\J$ is associated with a random sojourn time $T_j$ representing the total time spent by a request in the edge server, including both queueing and service components. The specific structure of $T_j$ depends on the adopted \emph{service regime}, allowing different levels of modeling fidelity within a unified framework.

To represent processing heterogeneity between cache hits and misses, we define for each edge server $j\in\J$, exponential service rates
$\mu^\vartheta_j>0$ for $\vartheta \in \Theta:=\{H, M\}$, corresponding to cache-hit ($H$) and cache-miss ($M$) requests, respectively.

For each demand point $i\in\I$, each request is independently classified as a cache hit or a cache miss with probability $p_i^\vartheta\in[0,1]$, for $\vartheta \in \Theta$, such that $\sum_{\vartheta \in \Theta} p^\vartheta_i = 1$.

Although the standard CDN setting distinguishes only between cache hits and cache misses,
we state the queueing model using the more general notation $\vartheta\in\Theta$. This
allows the same framework to accommodate multiple request classes, such as different
types of cache misses associated with heterogeneous content categories, retrieval
procedures, or origin layers. In this generalized interpretation, one element of $\Theta$,
denoted by $H$, represents cache hits, while the remaining elements correspond to
requests that require an upstream retrieval operation. The two-class case
$\Theta=\{H,M\}$ is recovered as the canonical CDN model considered throughout most of
the paper.

The CDN design problem combines two tightly interconnected classes of decisions:
location-allocation decisions (Figure~\ref{fig:loc}) and service dynamics.
\begin{figure}[h]
\centering
\begin{tikzpicture}[
    >=Latex,
    scale=0.85,
    region/.style    = {draw,rounded corners=3pt,fill=gray!5},
    thinarr/.style   = {->,thin},
    dashedarr/.style = {<-,line width=0.8pt,dashed},
    every node/.style = {align=center}
]

\def\EdgeW{0.6cm}
\def\EdgeH{0.8cm}
\def\OrigW{1cm}
\def\OrigH{1.1cm}

\tikzset{
    demand/.style={
        inner sep=0pt,
        outer sep=0pt,
        text=black,
        font=\fontsize{14}{14}\selectfont
    },
    edge server/.style={
        rectangle,
        minimum width=\EdgeW,
        minimum height=\EdgeH,
        inner sep=0pt,
        outer sep=0pt,
        draw=none,
        path picture={
            \node at (path picture bounding box.center)
            {\includegraphics[width=\EdgeW,height=\EdgeH]{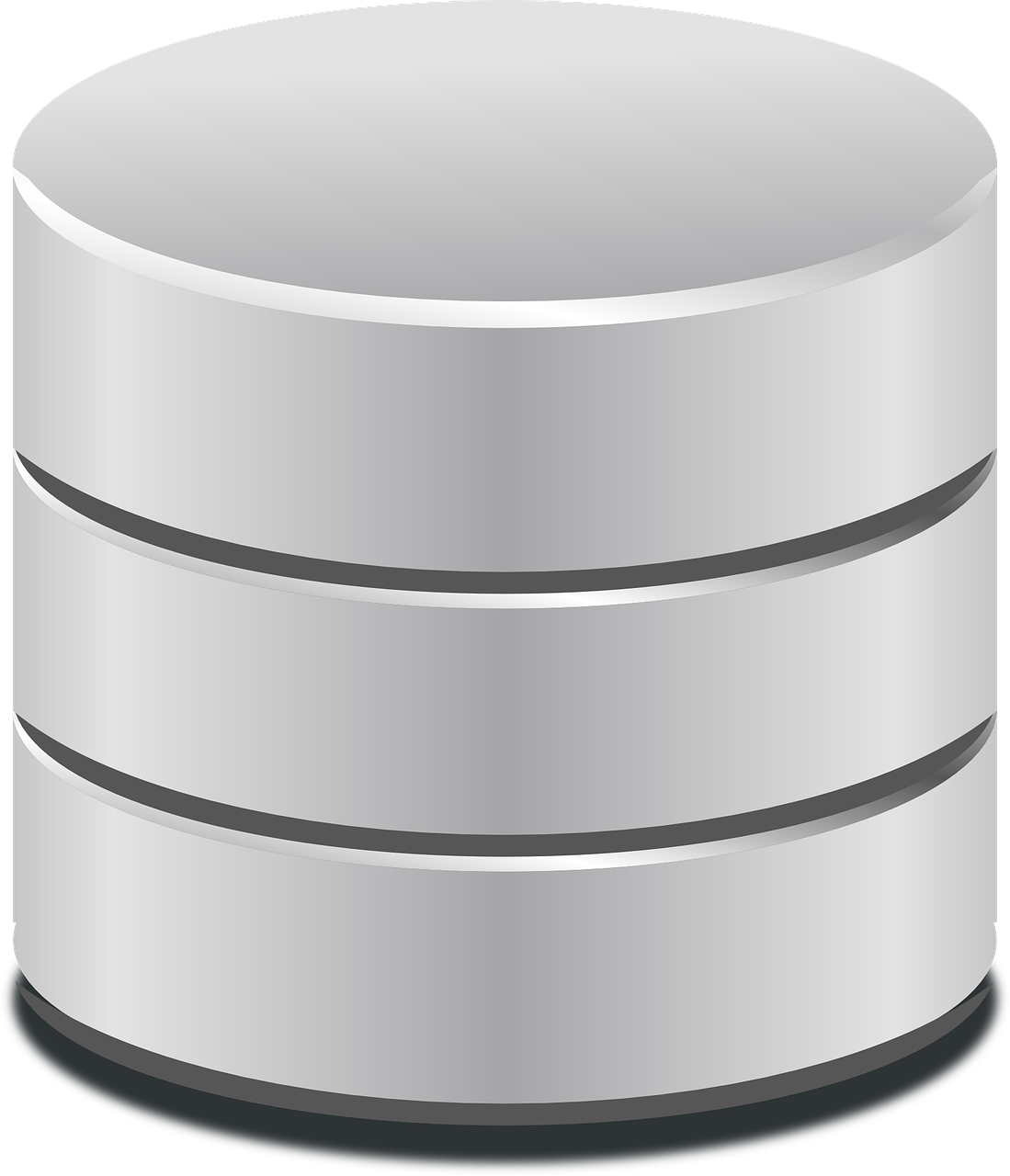}};
        }
    },
    origin server/.style={
        rectangle,
        minimum width=\OrigW,
        minimum height=\OrigH,
        inner sep=0pt,
        outer sep=0pt,
        draw=none,
        path picture={
            \node at (path picture bounding box.center)
            {\includegraphics[width=\OrigW,height=\OrigH]{server.png}};
            \fill[green!60,opacity=.33]
                (path picture bounding box.south west)
                rectangle
                (path picture bounding box.north east);
        }
    }
}

\draw[region] (-0.2,-6.8) rectangle (12.6,2.2);

\node[demand] (a1) at (1.0,1.4) {\small\faUser};
\node[demand] (a2) at (1.8,0.5) {\small\faUser};
\node[demand] (a3) at (2.2,-0.2) {\small\faUser};
\node[demand] (a4) at (1.2,-1.3) {\small\faUser};

\node[demand] (b1) at (4.2,1.2) {\small\faUser};
\node[demand] (b2) at (4.9,0.4) {\small\faUser};
\node[demand] (b3) at (5.2,-0.6) {\small\faUser};
\node[demand] (b4) at (4.0,-1.4) {\small\faUser};

\node[demand] (c1) at (7.1,1.5) {\small\faUser};
\node[demand] (c2) at (7.9,0.7) {\small\faUser};
\node[demand] (c3) at (8.4,-0.2) {\small\faUser};
\node[demand] (c4) at (7.3,-1.2) {\small\faUser};

\node[demand] (d1) at (10.2,1.1) {\small\faUser};
\node[demand] (d2) at (10.9,0.2) {\small\faUser};
\node[demand] (d3) at (11.2,-0.7) {\small\faUser};
\node[demand] (d4) at (9.8,-1.4) {\small\faUser};

\node[edge server] (f1) at (3.2,0.2) {};
\node[edge server] (f2) at (6.4,-0.6) {};
\node[edge server] (f3) at (9.3,0.6) {};

\node[fill=white,fill opacity=.55,text opacity=1,rounded corners=2pt,inner sep=1pt] at (f1.center) {$x_1$};
\node[fill=white,fill opacity=.55,text opacity=1,rounded corners=2pt,inner sep=1pt] at (f2.center) {$x_2$};
\node[fill=white,fill opacity=.55,text opacity=1,rounded corners=2pt,inner sep=1pt] at (f3.center) {$x_3$};

\node[demand] (a1b) at (1.0,-4.4) {\small\faUser};
\node[demand] (a2b) at (1.8,-4.5) {\small\faUser};
\node[demand] (a3b) at (2.2,-5.0) {\small\faUser};
\node[demand] (a4b) at (1.75,-5.8) {\small\faUser};

\node[demand] (b1b) at (4.2,-4.2) {\small\faUser};
\node[demand] (b2b) at (3.9,-5.4) {\small\faUser};
\node[demand] (b3b) at (5.2,-5.3) {\small\faUser};
\node[demand] (b4b) at (6.0,-6.0) {\small\faUser};

\node[demand] (c1b) at (7.1,-3.8) {\small\faUser};
\node[demand] (c2b) at (8.6,-5.7) {\small\faUser};
\node[demand] (c3b) at (7.8,-5.2) {\small\faUser};
\node[demand] (c4b) at (7.3,-6.2) {\small\faUser};

\node[demand] (d1b) at (10.2,-4.1) {\small\faUser};
\node[demand] (d2b) at (10.9,-5.2) {\small\faUser};
\node[demand] (d3b) at (11.2,-5.7) {\small\faUser};
\node[demand] (d4b) at (9.8,-6.4) {\small\faUser};

\node[edge server] (f1b) at (3.2,-4.2) {};
\node[edge server] (f2b) at (6.4,-4.6) {};
\node[edge server] (f3b) at (9.3,-4.6) {};

\node[fill=white,fill opacity=.55,text opacity=1,rounded corners=2pt,inner sep=1pt] at (f1b.center) {$x_4$};
\node[fill=white,fill opacity=.55,text opacity=1,rounded corners=2pt,inner sep=1pt] at (f2b.center) {$x_5$};
\node[fill=white,fill opacity=.55,text opacity=1,rounded corners=2pt,inner sep=1pt] at (f3b.center) {$x_6$};

\node[origin server] (o1) at (4,-2.5) {};
\node[origin server] (o2) at (8,-3) {};

\draw[thinarr] (a1.east) .. controls ($(a1.east)+(0.6,0)$) and ($(f1.north)+(-0.6,0.5)$) .. (f1.north);
\draw[thinarr] (a2.east) .. controls ($(a2.east)+(0.6,0)$) and ($(f1.west)+(-0.6,0.5)$) .. (f1.west);
\draw[thinarr] (a3.east) .. controls ($(a3.east)+(0.6,0)$) and ($(f1.west)+(-0.6,-0.5)$) .. (f1.west);
\draw[thinarr] (a4.east) .. controls ($(a4.east)+(0.6,0)$) and ($(f1.south)+(-0.6,-0.5)$) .. (f1.south);
\draw[thinarr] (b1.west) .. controls ($(b1.west)+(-0.2,0)$) and ($(f1.north)+(0.6,0.5)$) .. (f1.north);
\draw[thinarr] (b2.west) .. controls ($(b2.west)+(-0.2,0)$) and ($(f1.east)+(0.6,-0.5)$) .. (f1.east);

\draw[thinarr] (b3.east) .. controls ($(b3.east)+(0.6,0)$) and ($(f2.west)+(-0.6,0)$) .. (f2.west);
\draw[thinarr] (b4.east) .. controls ($(b4.east)+(0.6,0)$) and ($(f2.south)+(-0.6,-0.6)$) .. ($(f2.south)+(-0.1,0.0)$);
\draw[thinarr] (c4.west) .. controls ($(c4.west)+(-0.1,0)$) and ($(f2.south)+(0.6,-0.6)$) .. ($(f2.south)+(0.1,0.0)$);
\draw[thinarr] (c1.west) .. controls ($(c1.west)+(-0.6,0)$) and ($(f2.north)+(0,0.6)$) .. (f2.north);
\draw[thinarr] (c3.west) .. controls ($(c3.west)+(-0.6,0)$) and ($(f2.east)+(0.6,0.6)$) .. (f2.east);

\draw[thinarr] (c2.east) .. controls ($(c2.east)+(0.6,0)$) and ($(f3.west)+(-0.6,0)$) .. (f3.west);
\draw[thinarr] (d1.west) .. controls ($(d1.west)+(-0.6,0)$) and ($(f3.north)+(0,0.5)$) .. (f3.north);
\draw[thinarr] (d2.west) .. controls ($(d2.west)+(-0.6,0)$) and ($(f3.east)+(0.6,0)$) .. (f3.east);
\draw[thinarr] (d3.west) .. controls ($(d3.west)+(-0.6,0)$) and ($(f3.south)+(0.6,-0.5)$) .. ($(f3.south)+(0.1,0)$);
\draw[thinarr] (d4.north) .. controls ($(d4.north)+(-0.6,1)$) and ($(f3.south)+(0.6,-1)$) .. ($(f3.south)+(-0.1,0)$);

\draw[dashedarr] (o1.west)  .. controls ($(o1.west)+(-2,0)$)   and ($(f1.south)+(0.6,-1.5)$)  .. (f1.south);
\draw[dashedarr] (o1.north) .. controls ($(o1.north)+(0.1,0)$) and ($(f2.south)+(-0.5,-1.5)$) .. (f2.south);
\draw[dashedarr] (o2.north) .. controls ($(o2.north)+(0,2)$)   and ($(f3.south)+(-0.6,-1.5)$) .. ($(f3.south)+(-0.15,0)$);

\draw[dashedarr] (o1.south) .. controls ($(o1.south)+(-1,0)$)  and ($(f1b.north)+(0,1)$)      .. (f1b.north);
\draw[dashedarr] (o2.west)  .. controls ($(o2.west)+(-0.4,0)$) and ($(f2b.north)+(-0.3,1.5)$) .. (f2b.north);
\draw[dashedarr] (o2.east)  .. controls ($(o2.east)+(0.2,0)$)  and ($(f3b.north)+(0.2,1.5)$)  .. ($(f3b.north)+(-0.15,0)$);

\draw[thinarr] (a1b.east) .. controls ($(a1b.east)+(0.6,0)$) and ($(f1b.north)+(-0.6,0.5)$) .. (f1b.north);
\draw[thinarr] (a2b.east) .. controls ($(a2b.east)+(0.6,0)$) and ($(f1b.west)+(-0.6,0.5)$) .. (f1b.west);
\draw[thinarr] (a3b.east) .. controls ($(a3b.east)+(0.6,0)$) and ($(f1b.west)+(-0.6,-0.5)$) .. (f1b.west);
\draw[thinarr] (a4b.east) .. controls ($(a4b.east)+(0.6,0)$) and ($(f1b.south)+(-0.6,-0.5)$) .. (f1b.south);
\draw[thinarr] (b1b.west) .. controls ($(b1b.west)+(-0.2,0)$) and ($(f1b.north)+(0.6,0.5)$) .. (f1b.north);
\draw[thinarr] (b2b.west) .. controls ($(b2b.west)+(-0.2,0)$) and ($(f1b.east)+(0.6,-0.5)$) .. (f1b.east);

\draw[thinarr] (b3b.east) .. controls ($(b3b.east)+(0.6,0)$) and ($(f2b.west)+(-0.6,0)$) .. (f2b.west);
\draw[thinarr] (b4b.east) .. controls ($(b4b.east)+(0.6,0)$) and ($(f2b.south)+(-0.6,-0.6)$) .. ($(f2b.south)+(-0.1,0.0)$);
\draw[thinarr] (c4b.west) .. controls ($(c4b.west)+(-0.1,0)$) and ($(f2b.south)+(0.6,-0.6)$) .. ($(f2b.south)+(0.1,0.0)$);
\draw[thinarr] (c1b.west) .. controls ($(c1b.west)+(0.1,0)$) and ($(f2b.north)+(0,0.6)$) .. ($(f2b.north)+(0.1,0.0)$);
\draw[thinarr] (c3b.west) .. controls ($(c3b.west)+(-0.6,0)$) and ($(f2b.east)+(0.6,0.6)$) .. (f2b.east);

\draw[thinarr] (c2b.east) .. controls ($(c2b.east)+(0.6,0)$) and ($(f3b.west)+(-0.6,0)$) .. (f3b.west);
\draw[thinarr] (d1b.west) .. controls ($(d1b.west)+(-0.6,0)$) and ($(f3b.north)+(0,0.5)$) .. (f3b.north);
\draw[thinarr] (d2b.west) .. controls ($(d2b.west)+(-0.6,0)$) and ($(f3b.east)+(0.6,0)$) .. (f3b.east);
\draw[thinarr] (d3b.west) .. controls ($(d3b.west)+(-0.6,0)$) and ($(f3b.south)+(0.6,-0.5)$) .. ($(f3b.south)+(0.1,0)$);
\draw[thinarr] (d4b.north) .. controls ($(d4b.north)+(-0.6,1)$) and ($(f3b.south)+(0.6,-1)$) .. ($(f3b.south)+(-0.1,0)$);

\end{tikzpicture}
    \caption{CDN overview.\label{fig:loc}}
\end{figure}

In order to derive the expressions of the terms involved in the computation of the expected response times, we recall some useful and well-known results on queueing theory~\citep[see, e.g.,][]{gross2008fundamentals,kleinrock1975queueing,wolff1989stochastic}.

\begin{proposition}\label{prop:ET_general}
Consider an overall Poisson arrival process with total rate $\lambda>0$ that is partitioned into different streams with rates $\lambda^\vartheta \ge 0$, for $\vartheta \in \Theta$, i.e., $\lambda=\sum_{\vartheta\in \Theta} \lambda^\vartheta$.

\begin{enumerate}[label=\roman*), ref={\roman*}]
\item \label{prop:dsr}
Assume that the streams are processed by independent ${\rm M/M/1}$ queues, each with service rate $\mu^\vartheta>0$, and that $\lambda^\vartheta<\mu^\vartheta$ for all $\vartheta \in \Theta$. Then, the expected steady-state sojourn time of an arbitrary request is
$$
\mathbb{E}[T]
= \sum_{\vartheta \in \Theta} \frac{\lambda^\vartheta}{\lambda (\mu^\vartheta - \lambda^\vartheta)}.
$$

\item \label{prop:isr}
Assume that all requests are processed by a single-server queue with arrival rate $\lambda$ and service time $S$, where $S$ follows a mixture of exponential distributions $(\mathrm{H}_{|\Theta|})$ with probabilities $\lambda^\vartheta/\lambda$ and rates $\mu^\vartheta$, for $\vartheta \in \Theta$. If $\lambda \mathbb{E}[S] < 1$, then the expected steady-state sojourn time is
\begin{equation}\label{eq:pollaczek-khinchine}
    \mathbb{E}[T]
=
\mathbb{E}[S]
+
\frac{\lambda \mathbb{E}[S^2]}{2(1-\lambda \mathbb{E}[S])},
\end{equation}
where
$$
\mathbb{E}[S]
= \sum_{\vartheta \in\Theta} \frac{\lambda^\vartheta}{\lambda \mu^\vartheta}
\quad \text{and} \quad
\mathbb{E}[S^2]
= \frac{2}{\lambda} \sum_{\vartheta \in \Theta} \frac{\lambda^\vartheta}{(\mu^\vartheta)^2}.
$$
\end{enumerate}
\end{proposition}

\begin{proof}
We analyze each service structure separately.
\begin{enumerate}[label=\roman*)]
    \item 
Each stream $\vartheta \in \Theta$ is processed by an independent ${\rm M/M/1}$ queue with arrival rate $\lambda^\vartheta$ and service rate $\mu^\vartheta$. It is well known that, under the stability condition $\lambda^\vartheta < \mu^\vartheta$, the steady-state expected sojourn time in such a queue, denoted by $T^\vartheta$, is given by
$$
\mathbb{E}[T^\vartheta] = \frac{1}{\mu^\vartheta - \lambda^\vartheta}.
$$
An arbitrary request belongs to stream $\vartheta$ with probability $\lambda^\vartheta/\lambda$. Therefore, the overall expected sojourn time is obtained by conditioning on the stream type and applying the law of total expectation:
$$
\mathbb{E}[T]
= \sum_{\vartheta \in \Theta} \frac{\lambda^\vartheta}{\lambda} \,\mathbb{E}[T^\vartheta]
= \sum_{\vartheta \in \Theta} \frac{\lambda^\vartheta}{\lambda (\mu^\vartheta - \lambda^\vartheta)}.
$$

\item All requests are processed by a single-server queue with arrival rate $\lambda$ and service time $S$. The service time follows a mixture of exponential distributions. With probability $\lambda^\vartheta/\lambda$, the service time is exponentially distributed with rate $\mu^\vartheta$. Therefore, by conditioning on the type $\vartheta$, the first moment is
$$
\mathbb{E}[S]
= \sum_{\vartheta \in \Theta} \frac{\lambda^\vartheta}{\lambda} \cdot \frac{1}{\mu^\vartheta}
= \sum_{\vartheta \in\Theta} \frac{\lambda^\vartheta}{\lambda \mu^\vartheta}.
$$
Analogously, since the second moment of an exponential distribution with rate $\mu^\vartheta$ is $2/(\mu^\vartheta)^2$, by the linearity of expectation for mixture distributions, it follows that
$$
\mathbb{E}[S^2]
= \sum_{\vartheta \in \Theta} \frac{\lambda^\vartheta}{\lambda} \cdot \frac{2}{(\mu^\vartheta)^2}
= \frac{2}{\lambda} \sum_{\vartheta \in \Theta} \frac{\lambda^\vartheta}{(\mu^\vartheta)^2}.
$$
Under the stability condition $\lambda \mathbb{E}[S] < 1$, the classical Pollaczek--Khinchine formula for general ${\rm M/G/1}$ queues applied to $\rm G=\rm H_{|\Theta|}$ completes the proof.
\end{enumerate}
\end{proof}

\subsection{Service Dynamics}\label{sec:service_dynamics}

Within this framework, we consider two alternative representations of the service dynamics at each edge server, based on the way the service regimes are aggregated.

The first one, referred to as the disaggregated service regime (DSR, for short), models different request classes as independent queueing subsystems. The second one, referred to as the integrated service regime (ISR, for short), models all request classes through a common shared queue with a mixed service-time distribution. These two regimes provide complementary levels of aggregation and lead to different congestion expressions.

We summarize the assumptions that govern the arrival and service processes in the queueing-based CDN.

\begin{enumerate}[label=(A\arabic*), ref={\rm A\arabic*}]
\item \label{a1}
Allocation decisions are deterministic and independent of the queue state. Each demand point $i$ is assigned to exactly one edge server $j\in\J$, and each edge server $j$ is assigned to exactly one origin server $k\in\K$.

\item \label{a2}
For each demand point $i\in\I$, requests are generated according to independent Poisson processes with rate $\xi_i>0$.

\item \label{a3}
For each demand point $i\in\I$, each request is independently classified into cache type $\vartheta \in \Theta$ with probability $p_i^\vartheta \in [0,1]$, where $\sum_{\vartheta \in \Theta} p_i^\vartheta = 1$. In the classical setting, $\Theta=\{H,M\}$ corresponds to cache hits and cache misses.

\item \label{a4}
For each edge server $j\in\J$, the service times associated with requests of type $\vartheta \in \Theta$ are independent and exponentially distributed with rates $\mu^\vartheta_j>0$. In the classical setting, $\Theta=\{H,M\}$ corresponds to cache hits and cache misses.
\end{enumerate}

\begin{lemma}\label{lemma:arrivals}
   Given an allocation of demand points to edge servers, denote by $\I_j\subseteq \I$ the set of demand points assigned to server $j \in \J$. Under assumptions \ref{a1}--\ref{a3}, the following claims hold.

\begin{enumerate}[label=\roman*), ref=\roman*]

\item The arrival process at each edge server $j$ is Poisson with rate
$$
\Lambda_j = \sum_{i\in\I_j}\xi_i.
$$

\item The cache arrival processes at each edge server $j\in\J$ are independent Poisson processes with rates
$$
\Lambda^\vartheta_j=\sum_{i\in\I_j}\xi_i p_i^\vartheta,
\qquad \forall \vartheta\in\Theta.
$$
\end{enumerate}
\end{lemma}

\begin{proof}
    Assumptions \ref{a1} and \ref{a2} ensure that every demand point $i\in\I_j$, and only these points, generates requests to edge server $j$ according to independent Poisson processes with rates $\xi_i$. By superposition of independent Poisson processes, the aggregate arrival rate to edge server $j$ is $\Lambda_j = \sum_{i\in\I_j}\xi_i$, proving i).

    By assumption \ref{a3}, every demand point $i\in\I_j$ generates requests of type $\vartheta$ according to a thinned Poisson process with rate $\xi_i p^\vartheta_i$. The thinning property of Poisson processes and the superposition of independent Poisson processes imply ii).
\end{proof}
In what follows, we describe the structural and operational characteristics of the two queueing regimes considered in this paper.

\noindent \textbf{Disaggregated service regime (DSR).} Under DSR, cache hits and cache misses are processed through independent service subsystems.

\begin{theorem}\label{lemma:dsr}
Under assumptions \ref{a1}--\ref{a4} and a disaggregated service regime by cache status in $\Theta$. Given any edge server $j\in\J$. If $
\Lambda^\vartheta_j < \mu^\vartheta_j$
for every $\vartheta \in \Theta$, the service system at $j$ decomposes into $|\Theta|$ independent $\mathrm{M}/\mathrm{M}/1$ queues and the random sojourn time $T_j$ is hyperexponential with expectation
\begin{equation}\label{eq:expected_dsr}
    \mathbb{E}[T_j] = \sum_{\vartheta \in \Theta}
\frac{\Lambda_j^\vartheta}{\Lambda_j(\mu_j^\vartheta-\Lambda_j^\vartheta)}.
\end{equation}
\end{theorem}

\begin{proof}
By Lemma~\ref{lemma:arrivals}, the arrival flow is split into $|\Theta|$ independent Poisson streams with rates $\Lambda_j^\vartheta$. Attaching \ref{a4} and the disaggregated service regime, it follows that each stream is served by an $\mathrm{M}/\mathrm{M}/1$ queue with service rates $\mu_j^\vartheta$, with $\vartheta \in\Theta$. Finally, the expected sojourn time is derived from Proposition \ref{prop:ET_general}.\ref{prop:dsr}.
\end{proof}

Figure~\ref{fig:queue_dsr} illustrates DSR at an edge server $j$. The demand points assigned to $j$, collected in $\mathcal I_j$, generate independent Poisson request streams with rates $\xi_i$, which aggregate into a total arrival flow $\Lambda_j$. Upon reaching the edge server, each request is classified by the cache layer into one of the request types $\vartheta\in\Theta$. In the standard two-class setting, these are cache hits, with probability $p_i^H=p_i$, and cache misses, with probability $p_i^M=1-p_i$. This classification splits the incoming flow into independent streams with rates $\Lambda_j^\vartheta$, $\vartheta\in\Theta$. Under DSR, these streams are processed in separate queueing subsystems, each modeled as an $\mathrm{M}/\mathrm{M}/1$ system with parameters $(\Lambda_j^\vartheta,\mu_j^\vartheta)$. Miss requests additionally trigger a retrieval operation from the assigned origin server $k$, incurring an upstream delay of $\kappa_2 \D(x_j,b_k)$, whereas all requests are subject to the user-to-edge access delay $\kappa_1 \D(a_i,x_j)$.
Thus, DSR captures heterogeneous congestion effects through class-specific queues, while the coupling among classes is induced by common location-allocation decisions and by the shared origin-assignment structure.

\begin{figure}[h]
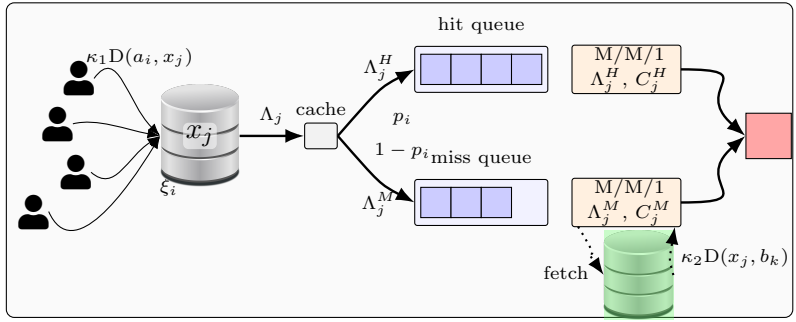

\centering

\begin{tikzpicture}[
    >=Latex,
    scale=0.8,
    inqueue/.style   = {draw,fill=blue!20,minimum width=4mm,minimum height=4mm},
    queuebox/.style  = {draw,fill=blue!5,rounded corners=1pt},
    modebox/.style   = {draw,fill=orange!12,rounded corners=1pt},
    label/.style     = {font=\tiny},
    thinarr/.style   = {->,thin},
    thickarr/.style  = {->,line width=0.9pt},
    dashedarr/.style = {->,line width=0.8pt,dotted},
    every node/.style = {align=center},
    region/.style    = {draw,rounded corners=3pt,fill=gray!4},
]

\def\EdgeW{1.05cm}
\def\EdgeH{1.35cm}
\def\OrigW{0.92cm}
\def\OrigH{1.18cm}

\tikzset{
    demand/.style={
        inner sep=0pt,
        outer sep=0pt,
        text=black,
        font=\fontsize{13}{13}\selectfont
    },
    facilityimg/.style={
        rectangle,
        minimum width=\EdgeW,
        minimum height=\EdgeH,
        inner sep=0pt,
        outer sep=0pt,
        draw=none,
        path picture={
            \node at (path picture bounding box.center)
            {\includegraphics[width=\EdgeW,height=\EdgeH]{server.png}};
        }
    },
    originimg/.style={
        rectangle,
        minimum width=\OrigW,
        minimum height=\OrigH,
        inner sep=0pt,
        outer sep=0pt,
        draw=none,
        path picture={
            \node at (path picture bounding box.center)
            {\includegraphics[width=\OrigW,height=\OrigH]{server.png}};
            \fill[green!60,opacity=.33]
                (path picture bounding box.south west)
                rectangle
                (path picture bounding box.north east);
        }
    }
}

\draw[region] (-2.2,-3) rectangle (10.8,2.2);

\coordinate (D0)      at (0,0);
\coordinate (Router)  at (1,0);
\coordinate (SplitHM) at (3,0);
\coordinate (HitQ)    at (6.75,1.1);
\coordinate (MissQ)   at (6.75,-1.1);
\coordinate (Origin)  at (8.25,-2.3);
\coordinate (Out)     at (10.4,0);

\node[demand] (d1) at ($(D0)+(-1.0,0.95)$) {\faUser};
\node[demand] (d2) at ($(D0)+(-1.35,0.20)$) {\faUser};
\node[demand] (d3) at ($(D0)+(-1.05,-0.60)$) {\faUser};
\node[demand] (d4) at ($(D0)+(-1.75,-1.25)$) {\faUser};

\node[label,anchor=west,xshift=5pt] at ($(D0)+(-0.05,-0.85)$) {$\xi_i$};

\node[label] at ($(D0)+(0,1.3)$) {$\kappa_1\D(a_i,x_j)$};

\node[facilityimg] (router) at (Router) {};
\node[
    fill=white,
    fill opacity=.55,
    text opacity=1,
    rounded corners=2pt,
    inner sep=1pt
] at (router.center) {$x_j$};

\foreach \k/\y in {1/0.55, 2/0.10, 3/-0.35, 4/-0.90}{
  \draw[thinarr] (d\k.east) .. controls ($(d\k.east)+(0.4,\y)$) and ($(router.west)+(-0.55,\y)$) .. (router.west);
}

\node[draw,rounded corners=1pt,fill=gray!10,inner sep=4pt,minimum size=3.4mm] (split) at (SplitHM) {\hspace*{0.15cm}};
\draw[thickarr] (router.east) -- node[midway,above,label] {$\Lambda_{j}$} (split.west);

\node[label,above,yshift=4pt] at (SplitHM) {cache};
\node[label,below,xshift=25pt] at ($(SplitHM)+(0.25,0.55)$) {$p_{i}$};
\node[label,above,xshift=25pt] at ($(SplitHM)+(0.25,-0.55)$) {$1-p_i$};

\draw[queuebox] ($(HitQ)+(-2.2,-0.38)$) rectangle ($(HitQ)+(0,0.38)$);
\node[label] at ($(HitQ)+(-1.1,0.72)$) {hit queue};
\foreach \x in {-1.85,-1.35,-0.85,-0.35}{
  \node[inqueue] at ($(HitQ)+(\x,0)$) {};
}
\draw[modebox] ($(HitQ)+(0.4,-0.40)$) rectangle ($(HitQ)+(2.2,0.40)$);
\node[label] at ($(HitQ)+(1.35,0)$) {$\mathrm M/\mathrm M/1$\\$\Lambda^H_{j},\, C^H_{j}$};

\draw[queuebox] ($(MissQ)+(-2.2,-0.38)$) rectangle ($(MissQ)+(0,0.38)$);
\node[label] at ($(MissQ)+(-1.1,0.72)$) {miss queue};
\foreach \x in {-1.85,-1.35,-0.85}{
  \node[inqueue] at ($(MissQ)+(\x,0)$) {};
}
\draw[modebox] ($(MissQ)+(0.4,-0.40)$) rectangle ($(MissQ)+(2.2,0.40)$);
\node[label] at ($(MissQ)+(1.35,0)$) {$\mathrm M/\mathrm M/1$\\$\Lambda^M_{j},\, C^M_{j}$};

\draw[thickarr]
  (split.east)
  .. controls ($(split.east)+(0.7,0.8)$)
  .. node[midway,above,label] {$\Lambda^H_{j}$}
  ($(HitQ)+(-2.2,0)$);

\draw[thickarr]
  (split.east)
  .. controls ($(split.east)+(0.7,-0.8)$)
  .. node[midway,below,label] {$\Lambda^M_{j}$}
  ($(MissQ)+(-2.2,0)$);

\node[originimg] (origin) at (Origin) {};

\draw[dashedarr]
  ($(MissQ)+(0.5,-0.4)$)
  .. controls ($(MissQ)+(1,-1.2)$) and ($(Origin)+(-1,0.6)$)
  .. node[midway,right,label,below,xshift=-10pt] {fetch}
  (origin.west);

\draw[dashedarr]
  (origin.east)
  .. controls ($(Origin)+(0.5,0.3)$) and ($(MissQ)+(2,-1)$)
  .. node[midway,right,label] {$\kappa_2\D(x_j,b_k)$}
  ($(MissQ)+(2.1,-0.4)$);

\node[inqueue,minimum size=6mm,fill=red!35] (end) at (Out) {};

\draw[thickarr]
  ($(HitQ)+(2.2,0)$)
  .. controls ($(HitQ)+(3.6,0.0)$) and ($(end)+(-1.3,0.8)$)
  .. (end.west);

\draw[thickarr]
  ($(MissQ)+(2.2,0)$)
  .. controls ($(MissQ)+(3.6,0.0)$) and ($(end)+(-1.3,-0.8)$)
  .. (end.west);

\end{tikzpicture}
\caption{Diagram of edge server operation under DSR.\label{fig:queue_dsr}}
\end{figure}

\noindent\textbf{Integrated service regime (ISR).} Under ISR, all requests are processed within a single shared queue. Cache hits and misses correspond to different service requirements within the same system, yielding a mixed service-time distribution.

\begin{theorem}\label{lemma:isr}
Under assumptions \ref{a1}--\ref{a4} and an integrated service regime. Given any edge server $j\in\J$. If $
\sum_{\vartheta \in \Theta} \frac{\Lambda^\vartheta_j}{\mu^\vartheta_j} < 1$, the service system at $j$ follows an $\mathrm{M}/\mathrm{H}_{|\Theta|}/1$ queue and the expectation of the sojourn time $T_j$ is
\begin{equation}\label{eq:expected_isr}
    \E[T_j]=
    \frac{1}{\Lambda_j} \sum_{\vartheta\in \Theta} \frac{\Lambda_j^\vartheta}{\mu^\vartheta_j}
    +
\frac{ \sum_{\vartheta \in \Theta} \frac{\Lambda_j^\vartheta}{(\mu^\vartheta_j)^2}}{
1- \sum_{\vartheta\in \Theta} \frac{\Lambda_j^\vartheta}{\mu^\vartheta_j}}.
\end{equation}
\end{theorem}

\begin{proof}
    By Lemma~\ref{lemma:arrivals}, the overall Poisson arrival flow is split into $|\Theta|$ independent Poisson streams with rates $\Lambda_j^\vartheta$. Attaching \ref{a4} and the integrated service regime, it follows that the service time is a mixture of exponential distributions with probabilities $\Lambda_j^\vartheta/\Lambda_j$ and rates $\mu_j^\vartheta$, for $\vartheta \in \Theta$, leading to an $\mathrm{M}/\mathrm{H}_{|\Theta|}/1$ queue. Finally, the expected sojourn time is derived from Proposition \ref{prop:ET_general}.\ref{prop:isr}.
\end{proof}

DSR provides a structured representation of heterogeneous service regimes, while ISR captures the interaction between cache hits and misses through shared congestion effects.

Figure~\ref{fig:queue_isr} illustrates ISR at an edge server $j$. As in the previous setting, the demand points in $\mathcal I_j$ generate independent Poisson flows with rates $\xi_i$, which aggregate into a total arrival rate $\Lambda_j$ at the edge server. Upon arrival, each request is classified by the cache mechanism into one of the types $\vartheta\in\Theta$. In the standard case, this corresponds to a hit or a miss with probabilities $p_i^H = p_i$ and $p_i^M=1-p_i$, respectively. However, in contrast to DSR, this classification does not induce separate queueing subsystems. Instead, all requests join a single shared queue and are processed by a common service facility modeled as an $\mathrm{M}/\mathrm{H}_{|\Theta|}/1$ system. The service time of each request depends on its type, leading to a mixed service-time distribution where type-$\vartheta$ requests are served with exponential rates $\mu_j^\vartheta$. Consequently, congestion effects are shared across all requests, and the delay experienced by one class is influenced by the presence of the others. Additionally, miss requests trigger a retrieval operation from the assigned origin server $k$, incurring an upstream delay of $\kappa_2 \D(x_j,b_k)$, while all requests incur the access delay $\kappa_1 \D(a_i,x_j)$. 
Thus, ISR aggregates heterogeneous service requirements into a single congestion system and is therefore more restrictive in terms of shared capacity effects than DSR.

\begin{figure}[h]
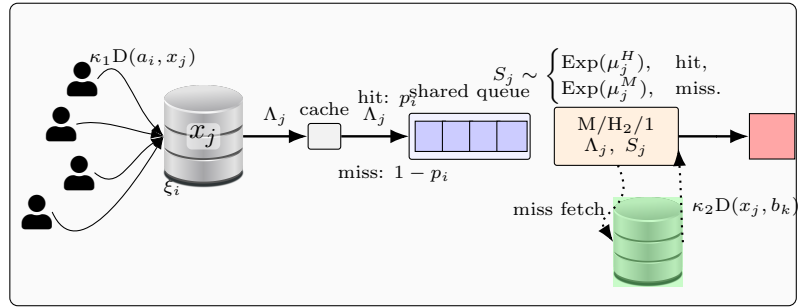

\centering

\begin{tikzpicture}[
    >=Latex,
    scale=0.8,
    inqueue/.style   = {draw,fill=blue!20,minimum width=4mm,minimum height=4mm},
    queuebox/.style  = {draw,fill=blue!5,rounded corners=1pt},
    modebox/.style   = {draw,fill=orange!12,rounded corners=1pt},
    label/.style     = {font=\tiny},
    thinarr/.style   = {->,thin},
    thickarr/.style  = {->,line width=0.9pt},
    dashedarr/.style = {->,line width=0.8pt,dotted},
    every node/.style = {align=center},
    region/.style    = {draw,rounded corners=3pt,fill=gray!4},
]

\def\EdgeW{1.05cm}
\def\EdgeH{1.35cm}
\def\OrigW{0.92cm}
\def\OrigH{1.18cm}

\tikzset{
    demand/.style={
        inner sep=0pt,
        outer sep=0pt,
        text=black,
        font=\fontsize{13}{13}\selectfont
    },
    facilityimg/.style={
        rectangle,
        minimum width=\EdgeW,
        minimum height=\EdgeH,
        inner sep=0pt,
        outer sep=0pt,
        draw=none,
        path picture={
            \node at (path picture bounding box.center)
            {\includegraphics[width=\EdgeW,height=\EdgeH]{server.png}};
        }
    },
    originimg/.style={
        rectangle,
        minimum width=\OrigW,
        minimum height=\OrigH,
        inner sep=0pt,
        outer sep=0pt,
        draw=none,
        path picture={
            \node at (path picture bounding box.center)
            {\includegraphics[width=\OrigW,height=\OrigH]{server.png}};
            \fill[green!60,opacity=.33]
                (path picture bounding box.south west)
                rectangle
                (path picture bounding box.north east);
        }
    }
}

\draw[region] (-2.2,-2.8) rectangle (10.8,2.2);

\coordinate (D0)      at (0,0);
\coordinate (Router)  at (1,0);
\coordinate (SplitHM) at (3,0);
\coordinate (SharedQ) at (6.4,0);
\coordinate (Origin)  at (8.35,-1.75);
\coordinate (Out)     at (10.4,0);

\node[demand] (d1) at ($(D0)+(-1.0,0.95)$) {\faUser};
\node[demand] (d2) at ($(D0)+(-1.35,0.20)$) {\faUser};
\node[demand] (d3) at ($(D0)+(-1.05,-0.60)$) {\faUser};
\node[demand] (d4) at ($(D0)+(-1.75,-1.25)$) {\faUser};

\node[label,anchor=west, xshift=5pt] at ($(D0)+(-0.05,-0.85)$) {$\xi_i$};

\node[label] at ($(D0)+(0,1.3)$) {$\kappa_1\D(a_i,x_j)$};

\node[facilityimg] (router) at (Router) {};
\node[
    fill=white,
    fill opacity=.55,
    text opacity=1,
    rounded corners=2pt,
    inner sep=1pt
] at (router.center) {$x_j$};

\foreach \k/\y in {1/0.55, 2/0.10, 3/-0.35, 4/-0.90}{
  \draw[thinarr] (d\k.east) .. controls ($(d\k.east)+(0.4,\y)$) and ($(router.west)+(-0.55,\y)$) .. (router.west);
}

\node[draw,rounded corners=1pt,fill=gray!10,inner sep=4pt,minimum size=3.4mm] (split) at (SplitHM) {\hspace*{0.15cm}};
\draw[thickarr] (router.east) -- node[midway,above,label] {$\Lambda_j$} (split.west);
\node[label,above,yshift=4pt] at (SplitHM) {cache};

\node[label] at ($(SplitHM)+(1.05,0.62)$) {hit: $p_i$};
\node[label] at ($(SplitHM)+(1.15,-0.62)$) {miss: $1-p_i$};

\draw[queuebox] ($(SharedQ)+(-2.0,-0.38)$) rectangle ($(SharedQ)+(0,0.38)$);
\node[label] at ($(SharedQ)+(-1.0,0.72)$) {shared queue};
\foreach \x in {-1.65,-1.20,-0.75,-0.30}{
  \node[inqueue] at ($(SharedQ)+(\x,0)$) {};
}

\draw[modebox] ($(SharedQ)+(0.4,-0.48)$) rectangle ($(SharedQ)+(2.45,0.48)$);
\node[label] at ($(SharedQ)+(1.42,0)$) {$\mathrm M/\mathrm H_2/1$\\$\Lambda_j,\; S_j$};

\draw[thickarr]
  (split.east)
  -- node[midway,above,label] {$\Lambda_j$}
  ($(SharedQ)+(-2.0,0)$);

\node[label] at ($(SharedQ)+(1.3,1.0)$)
{$S_j\sim
\begin{cases}
\mathrm{Exp}(\mu_j^H), & \text{hit},\\
\mathrm{Exp}(\mu_j^M), & \text{miss}.
\end{cases}$};

\node[originimg] (origin) at (Origin) {};

\draw[dashedarr]
  ($(SharedQ)+(1.45,-0.48)$)
  .. controls ($(SharedQ)+(1.8,-1.2)$) and ($(Origin)+(-0.9,0.5)$)
  .. node[midway,left,label] {miss fetch}
  (origin.west);

\draw[dashedarr]
  (origin.east)
  .. controls ($(Origin)+(0.55,0.25)$) and ($(SharedQ)+(2.5,-0.95)$)
  .. node[midway,right,label] {$\kappa_2\D(x_j,b_k)$}
  ($(SharedQ)+(2.45,-0.20)$);

\node[inqueue,minimum size=6mm,fill=red!35] (end) at (Out) {};

\draw[thickarr]
  ($(SharedQ)+(2.45,0)$)
  .. controls ($(SharedQ)+(3.6,0)$) and ($(end)+(-1.3,0)$)
  .. (end.west);

\end{tikzpicture}
\caption{Diagram of edge server operation under ISR.\label{fig:queue_isr}}
\end{figure}

Together with \ref{a1}--\ref{a4} on the arrival and service processes, the following regime-dependent stability conditions complete the system of assumptions for the underlying queueing dynamics of the CDN.

\begin{enumerate}[label=(A\arabic*), ref={\rm A\arabic*}, start=5]

\item \label{a5}\textbf{(Regime-dependent stability conditions)}
\begin{enumerate}[label=(A5.\alph*), ref= A5.\alph*]
\item \label{a5a}
Under an independent service structure (DSR), stability requires
$$
\Lambda^\vartheta_j < \mu^\vartheta_j,
\qquad \forall \vartheta \in \Theta,\ \forall j\in\J.
$$

\item \label{a5b}
Under a shared service structure (ISR), stability requires
$$
\sum_{\vartheta \in \Theta} \frac{\Lambda^\vartheta_j}{\mu^\vartheta_j} < 1,
\qquad \forall j\in\J.
$$

\end{enumerate}

\end{enumerate}

\begin{theorem}\label{thm:cdn_service_regimes}
Let $R_i$ be the random response time of a request generated at demand point $i \in \I$ within a {\rm CDN} under assumptions \ref{a1}--\ref{a5}. In both regimes, the expected response time is given by
\begin{equation}\label{eq:expected_time}
    \mathbb{E}[R_i]
=
\kappa_1 \D(a_i,x_{j(i)})+
\mathbb{E}[T_{j(i)}]
+ \kappa_2 \left(1-
\frac{\Lambda_{j(i)}^H}{\Lambda_{j(i)}}\right)  \D(x_{j(i)},b_{k(j(i))}),
\end{equation}
where $j(i)\in\J$ and $k(j(i))\in\K$ stand for the edge server and origin server assigned to demand point $i$, respectively.

The expressions for the expectation $\E[T_j]$ are \eqref{eq:expected_dsr} and \eqref{eq:expected_isr} under {\rm DSR} and {\rm ISR}, respectively.
\end{theorem}

\begin{proof}
    Proceeding backward, the last regime-dependent claim about the expectation of the sojourn time $T_j$ is a corollary of Theorems \ref{lemma:dsr} and \ref{lemma:isr}, respectively. On the other hand, the set of assumptions enables the decomposition of the random response time of a request $R_i$ into three separate terms. First, a deterministic time induced by the delay of communication from the demand point to its own single assigned edge server (\ref{a1}), which is given by the distance $\D(a_i,x_{j(i)})$ times the propagation delay per unit length $\kappa_1$. Second, the stochastic sojourn time at server $j(i)$, $T_{j(i)}$, induced by the congestion and its service dynamics. Third, the stochastic time induced by the delay of communication from the edge server to its own single assigned origin server (\ref{a1}) in case the request is not a cache hit. The latter component follows a Bernoulli distribution, which takes the value $0$ with probability $\frac{\Lambda^H_{j(i)}}{\Lambda_{j(i)}}$ and $\kappa_2\D(x_{j(i)},b_{k(j(i))})$ otherwise. Hence, $R_i$ equals
    \begin{equation*}
        R_i = \kappa_1 \D(a_i,x_{j(i)}) + T_{j(i)} + \kappa_2\D(x_{j(i)},b_{k(j(i))})B_i,
    \end{equation*}
    where $B_i\sim {\rm Bernoulli}\left(1-
\frac{\Lambda_{j(i)}^H}{\Lambda_{j(i)}}\right)$. Thus, Equation~\eqref{eq:expected_time} follows from the law of total expectation.
\end{proof}
Theorem~\ref{thm:cdn_service_regimes} provides the response-time expression that will be embedded in the optimization models developed in the next sections. The deterministic distance terms account for access and origin-retrieval delays, while the term $\mathbb{E}[T_j]$ contains the nonlinear congestion structure. It is precisely this latter component, together with the stability conditions, that motivates the conic lifting framework developed in the sequel.

In the following sections, we integrate the structure of the response-time expressions and the associated stability conditions to derive optimization models for the design of CDN architectures under the different service regimes. For the sake of clarity, we first focus on a simplified setting consisting of a single edge server.

In this setting, even when several origin servers are available, the absence of edge-server assignment decisions leads to a continuous conic optimization problem. This single-edge-server case isolates the core nonlinear queueing and location-capacity structure of the model and allows us to derive fully conic reformulations of the resulting optimization problem. These reformulations represent the nonlinear expressions appearing in the objective function and the constraints through suitable conic liftings, and they can be solved using standard conic optimization solvers such as MOSEK. In this reduced setting, the problem also exhibits strong connections with continuous facility location models of Weber type.

Building upon these reformulations, we then extend the analysis to the general setting with multiple edge servers and multiple origin servers. This leads to a mixed-integer conic optimization formulation that integrates location, allocation, origin-assignment, and capacity decisions within a unified framework, while preserving the conic structure of the underlying queueing-performance expressions.

\section{Single-Server CDN Design via Conic Optimization}\label{sec:singleServer}

In this section, we first analyze the CDN design problem in case a single edge server is to be designed. We provide a conic optimization framework for the problem that will set the basis for capturing the nonlinearity that encodes the system accuracy. It leads to the final discrete conic optimization model as an exact approach to the very general multi-server CDN design problem derived throughout this paper. 

Given a finite set of demand points $\{a_i:i\in\I\}\subset \R^m$ and also a location of single origin server $b\in\R^m$, the aim of the single-server CDN design problem reduces to find the coordinates of a placement $x\in\R^m$ and a stable and bounded vector of cache service rates $\mu=\left(\mu^\vartheta\right)_{\vartheta\in \Theta}\in\R_+^{|\Theta|}$ of the edge server that minimize a loss function of the expected response times, $(\E[R_i(x,\mu)])_{i\in\I}$. Notice that in this simpler case, the server-dependency in the arrival flows $\Lambda_j$ and $\Lambda_j^\vartheta$ (Lemma \ref{lemma:arrivals}) is not needed, thereby we can omit it and denote $\Lambda=\sum_{i\in\I}\xi_i$ and $\Lambda^\vartheta=\sum_{i\in\I}\xi_ip^\vartheta_i$ for all $\vartheta\in\Theta$. By Theorem~\ref{thm:cdn_service_regimes}, the expected response time vector is given, componentwise (per each demand), by
\begin{equation}
    \E[R_i(x,\mu)] = \kappa_1 \D(a_i,x)+ \E[T(\mu)] + \kappa_2\left(1-\frac{\Lambda^H}{\Lambda}\right) \D(x,b).
\end{equation}

We assume that the vector of service rates, in addition to satisfying the regime-dependent stability conditions \ref{a5a} or \ref{a5b}, is restricted to lie in a bounded region $\mathcal{C} \subset \R^{|\Theta|}_+$. This assumption naturally allows for the inclusion of an \emph{efficiency-oriented bound}, which is meaningful in the context of CDN design. Specifically, let $c^\vartheta >0$, for $\vartheta \in \Theta$, denote the unit cost associated with the installation of the edge server with service rate of type $\vartheta$, and let $\Gamma>0$ represent the available budget for capacity investment. Then, the feasible region $\mathcal{C}$ can be defined as
\begin{equation}
\mathcal{C} = \left\{ (\mu^\vartheta)_{\vartheta\in \Theta} \in \R^{|\Theta|}_+ : \sum_{\vartheta \in \Theta} c^\vartheta \mu^\vartheta \le \Gamma \right\}.
\label{budget}\tag{Budget}
\end{equation}

The overall feasible region for the service rates is then denoted by $\mathcal{C}_{\rm stab}$ and is defined as the intersection of $\mathcal{C}$ with the region induced by the corresponding stability conditions.

The expected response times of the different demand users are aggregated by means of a conic representable, nondecreasing, and lower semicontinuous loss function $\Phi:\R^{|\I|}\to \R$. Thus, the single-server CDN design problem can be stated as follows:

\begin{equation}\label{1cdn}\tag{1-CDN}
    \min_{x\in\R^m, \mu\in \mathcal{C}_{\rm stab}} \Phi(\E[R_i(x,\mu)])_{i\in\I}.
\end{equation}

\begin{remark}
In the case where $\Phi$ is additive (at least on the nonnegative orthant), then \eqref{1cdn} can be decomposed as
\begin{equation*}
\begin{array}{ccc}
\displaystyle\min_{x\in\R^m} \ \Phi\left(\kappa_1 \D(a_i,x) + \kappa_2 \left(1-\frac{\Lambda^H}{\Lambda}\right) \D(x,b)\right)_{i\in\I}
& + &
\displaystyle\min_{\mu\in \mathcal{C}_{\rm stab}} \ \Phi\big(\E[T(\mu)]\big)_{i\in\I}.
\end{array}
\end{equation*}

In particular, if $\Phi$ is the sum operator, then the first problem above coincides with a variant of the classical Weber problem~\citep{weiszfeld1937point}, where $a_i$ for $i\in\I$ plus $b$ act as $|\I|+1$ demand points with weights $\kappa_1$ for the $a_i$-points and $\kappa_2|\I|\left(1-\frac{\Lambda^H}{\Lambda}\right)$ for the point $b$. This problem can be efficiently solved using modern conic optimization techniques~\citep{blanco2014revisiting}.

Furthermore, if the service rates are unbounded or the {\rm CDN} operates under negligible congestion, the second optimization problem above becomes trivially zero. In this case, the single-server {\rm CDN} design problem reduces to a Weber-type problem and can be solved accordingly. In our computational experiments, this framework is referred to as the {\rm \bf UNC} (uncongested) system.
\end{remark}

The core matter of our approach is to prove efficient conic liftings of the objective function and the stability conditions, which entail a unified conic programming reformulation for \eqref{1cdn} across the different regimes.\\

\noindent {\bf Stability of edge server.} The edge server operates as a service node processing incoming requests. To ensure stable operation of the system, the installed service capacity must exceed the incoming traffic and the cache dynamics.

Under DSR, cache hits and misses are processed in independent queueing subsystems. Stability requires that, for each cache status, the service capacity strictly exceeds the corresponding arrival rate. This is enforced through
\begin{align}
    \mu^\vartheta - \Lambda^\vartheta \ge \varepsilon,
    \qquad
    \vartheta \in \Theta,\label{stab:dsr}\tag{stab$_{\rm DSR}$}
\end{align}
where $\varepsilon > 0$ is a small safety margin that prevents the system from operating arbitrarily close to saturation.

Under ISR, all requests are processed within a single shared queueing system, and stability is governed by the traffic intensity $\sum_{\vartheta\in\Theta}\frac{\Lambda^\vartheta}{\mu^\vartheta}$, whose stability is ensured by imposing
\begin{align} \sum_{\vartheta\in \Theta} \frac{\Lambda^\vartheta}{\mu^\vartheta} \le 1 - \varepsilon. \label{stab:isr}\tag{stab$_{\rm ISR}$}
\end{align}

\begin{proposition}\label{prop:stab}
    The feasible set 
    \begin{equation*}
    \left\{(\mu,\phi)\in \R^{|\Theta|}_+\times\R^{|\Theta|}_+ :  (\mu^\vartheta, \phi^\vartheta,1)\in \SS^2_+\;  \forall \vartheta\in\Theta, \quad 1-\sum_{\vartheta\in\Theta} \Lambda^\vartheta\phi^\vartheta \geq \varepsilon\right\}
    \end{equation*}
    is a conic lifting of the feasible set of \eqref{stab:isr}.
\end{proposition}

\begin{proof}
    Let $\mu\in \R^{|\Theta|}_+$ that satisfy \eqref{stab:isr}, then it follows that $\left(\mu, \left(\frac{1}{\mu^\vartheta}\right)_{\vartheta\in\Theta}\right)\in \R^{|\Theta|}_+\times\R^{|\Theta|}_+$ belongs to the conic lifting. Conversely, let $(\mu,\phi)\in \R^{|\Theta|}_+\times\R^{|\Theta|}_+$ in the conic lifting, thereby
    \begin{equation*}
        \sum_{\vartheta\in \Theta} \frac{\Lambda^\vartheta}{\mu^\vartheta} \leq \sum_{\vartheta\in\Theta} \Lambda^\vartheta\phi^\vartheta \leq 1-\varepsilon,
    \end{equation*}
    where the former inequality holds by $(\mu^\vartheta, \phi^\vartheta,1)\in \SS^2_+\;  \forall \vartheta\in\Theta$ and the latter does by $1-\sum_{\vartheta\in\Theta} \Lambda^\vartheta\phi^\vartheta \geq \varepsilon$. Therefore, $\mu$ satisfies \eqref{stab:isr}.
\end{proof}

\noindent {\bf Distance modeling.} The placement of the edge server is determined within the optimization model, which means that the distances between infrastructure components cannot be precomputed. Instead, these distances must be modeled explicitly as functions of the decision variables representing server location.

\begin{proposition}\label{prop:distance}
    Let $\D$ be a distance induced by an $\ell_p$-norm (with $1\leq p \in \mathbb{Q}$), i.e., $\D(x,y)=\|x-y\|_p$. The feasible set 
    \begin{equation*}
    \left\{(x,d)\in \R^m\times\R_+ : (a-x,d)\in \LL^m_p \right\}
    \end{equation*}
    is a conic lifting of $\D(a,\cdot)$, for all $a\in \R^m$.
\end{proposition}

\begin{proof}
    The result follows directly from that the conic lifting is precisely the epigraph of $\D(a,\cdot)$.
\end{proof}

\noindent {\bf Expected sojourn time.} In what follows, we derive results on conic liftings of the expected sojourn times $\E[T(\mu)]$ under the two regimes  that we consider under congestion (DSR and ISR).

Under DSR, cache hits and cache misses are processed in independent queueing subsystems. The incoming flow is split into $|\Theta|$ streams with arrival rates $\Lambda^\vartheta$ and service rates $\mu^\vartheta$ for every $\vartheta\in\Theta$. 


\begin{theorem}\label{th:expectation_dsr}
    The feasible set 
    \begin{equation*}
    \left\{(\mu,\theta,\tau)\in \R^{|\Theta|}_+\times\R^{|\Theta|}_+\times \R_+ :  (\Lambda(\mu^\vartheta-\Lambda^\vartheta), \theta^\vartheta,1)\in \SS^2_+\;  \forall \vartheta\in\Theta, \quad \tau\geq \sum_{\vartheta\in\Theta} \Lambda^\vartheta\theta^\vartheta\right\}
    \end{equation*}
    is a conic lifting of $\E[T]$ under {\rm DSR} regime.
\end{theorem}

\begin{proof}
    Let $(\mu,\tau)\in \R^{|\Theta|}_+\times \R_+$ in the epigraph of $\E[T]$ under DSR regime. It means by Theorem \ref{lemma:dsr}
    \begin{equation*}
        \E[T(\mu)]= \frac{1}{\Lambda} \sum_{\vartheta\in\Theta} \frac{\Lambda^\vartheta}{\mu^\vartheta-\Lambda^\vartheta}\leq \tau,
    \end{equation*}
    then it follows that $\left(\mu, \left(\frac{1}{\Lambda(\mu^\vartheta-\Lambda^\vartheta)}\right)_{\vartheta\in\Theta},\tau\right)\in \R^{|\Theta|}_+\times\R^{|\Theta|}_+\times \R_+$ belongs to the conic lifting. Conversely, let $(\mu,\theta,\tau)\in \R^{|\Theta|}_+\times\R^{|\Theta|}_+\times \R_+$ in the conic lifting, thereby
    \begin{equation*}
        \E[T(\mu)]= \frac{1}{\Lambda} \sum_{\vartheta\in\Theta} \frac{\Lambda^\vartheta}{\mu^\vartheta-\Lambda^\vartheta} \leq \sum_{\vartheta\in\Theta} \Lambda^\vartheta\theta^\vartheta \leq \tau,
    \end{equation*}
    where the former inequality holds by $(\Lambda(\mu^\vartheta-\Lambda^\vartheta), \theta^\vartheta,1)\in \SS^2_+\;  \forall \vartheta\in\Theta$ and the latter does by $\tau\geq \sum_{\vartheta\in\Theta} \Lambda^\vartheta\theta^\vartheta$. Therefore, $(\mu,\tau)$ belongs to the epigraph of $\E[T]$ under DSR regime as it was claimed.
\end{proof}

Under ISR, all requests assigned to the edge server are processed within a single shared queueing system. Cache hits and cache misses correspond to different service requirements within the same server, so that the service-time distribution is a mixture of exponentials. Analogously to the above theorem, we can provide a conic lifting under an integrated framework.

\begin{theorem}\label{th:expectation_isr}
    The feasible set 
    \begin{align*}
    \bigg\{(\mu,\phi,\theta,\varrho,\tau)\in \R^{|\Theta|}_+\times\R^{|\Theta|}_+\times \R^{|\Theta|}_+\times \R^{|\Theta|}_+\times \R_+ : \; &  (\mu^\vartheta, \phi^\vartheta,1)\in \SS^2_+\;  \forall \vartheta\in\Theta, \\
    & \left(\Lambda\mu^\vartheta, \theta^\vartheta,1\right)\in \SS^2_+\;  \forall \vartheta\in\Theta, \\
    & \left(1-\sum_{\vartheta\in\Theta} \Lambda^\vartheta\phi^\vartheta, \varrho^\vartheta, \phi^\vartheta\right)\in \SS^2_+\; \forall \vartheta\in\Theta,\\
    & \tau \geq \sum_{\vartheta\in\Theta} \Lambda^\vartheta(\theta^\vartheta+\varrho^\vartheta) \Bigg\}
    \end{align*}
    is a conic lifting of $\E[T]$ under {\rm ISR} regime.
\end{theorem}

\begin{proof}
    Let $(\mu,\tau)\in \R^{|\Theta|}_+\times \R_+$ in the epigraph of $\E[T]$ under ISR regime. It means by Theorem \ref{lemma:isr}
    \begin{equation*}
        \E[T(\mu)]= \frac{1}{\Lambda} \sum_{\vartheta\in \Theta} \frac{\Lambda^\vartheta}{\mu^\vartheta}  +
\frac{ \sum_{\vartheta \in \Theta} \frac{\Lambda^\vartheta}{(\mu^\vartheta)^2}}{
1- \sum_{\vartheta\in \Theta} \frac{\Lambda^\vartheta}{\mu^\vartheta}}\leq \tau,
    \end{equation*}
    then it is not difficult to check that
    \begin{equation*}
        \left(\mu,  \left(\frac{1}{\mu^\vartheta}\right)_{\vartheta\in\Theta},
        \left(\frac{1}{\Lambda\mu^\vartheta}\right)_{\vartheta\in\Theta},\left(\frac{\frac{1}{(\mu^\vartheta)^2}}{1-\sum_{\vartheta\in\Theta}\frac{\Lambda^\vartheta}{\mu^\vartheta}}\right)_{\vartheta\in\Theta},\tau\right)\in\R^{|\Theta|}_+\times\R^{|\Theta|}_+\times \R^{|\Theta|}_+\times \R^{|\Theta|}_+\times \R_+
    \end{equation*}
    belongs to the conic lifting. Conversely, let $(\mu,\phi,\theta,\varrho,\tau)\in \R^{|\Theta|}_+\times\R^{|\Theta|}_+\times \R^{|\Theta|}_+\times \R^{|\Theta|}_+\times \R_+$ in the conic lifting, thereby
    \begin{equation}\label{eq:aux}
        \frac{\frac{1}{(\mu^\vartheta)^2}}{1-\sum_{\vartheta\in\Theta}\frac{\Lambda^\vartheta}{\mu^\vartheta}}\leq \frac{(\phi^\vartheta)^2}{1-\sum_{\vartheta\in\Theta}\Lambda^\vartheta\phi^\vartheta}\leq \varrho^\vartheta, \quad \forall \vartheta \in \Theta,
    \end{equation}
    where the former inequality holds by $(\mu^\vartheta, \phi^\vartheta,1)\in \SS^2_+\;  \forall \vartheta\in\Theta,$ and the former does by $\left(1-\sum_{\vartheta\in\Theta} \Lambda^\vartheta\phi^\vartheta, \varrho^\vartheta, \phi^\vartheta\right)\in \SS^2_+\; \forall \vartheta\in\Theta$. Finally,
    \begin{equation*}
        \E[T(\mu)]= \frac{1}{\Lambda} \sum_{\vartheta\in \Theta} \frac{\Lambda^\vartheta}{\mu^\vartheta}  +
\frac{ \sum_{\vartheta \in \Theta} \frac{\Lambda^\vartheta}{(\mu^\vartheta)^2}}{
1- \sum_{\vartheta\in \Theta} \frac{\Lambda^\vartheta}{\mu^\vartheta}} \leq \sum_{\vartheta\in\Theta} \Lambda^\vartheta(\theta^\vartheta+\varrho^\vartheta) \leq \tau,
    \end{equation*}
    where the former inequality holds by $\left(\Lambda\mu^\vartheta, \theta^\vartheta,1\right)\in \SS^2_+\;  \forall \vartheta\in\Theta$ plus \eqref{eq:aux}, and the latter does by $\tau \geq \sum_{\vartheta\in\Theta} \Lambda^\vartheta(\theta^\vartheta+\varrho^\vartheta)$. Therefore, $(\mu,\tau)$ belongs to the epigraph of $\E[T]$ under ISR regime as it was claimed.
\end{proof}

As a corollary of all the results above, we provide the following compact conic program (CP) for the single-server CDN design problem \eqref{1cdn} that is able to solve the three different paradigms discussed throughout this paper (UNC, DSR, and ISR).

\begin{align*}\label{eq:1-cdn-cp}\tag{\rm 1-CDN-CP}
    \minimize_{x\in\R^m, \mu\in \mathcal{C}} \quad & \Phi(r_i)_{i\in\I} && \\
    \st \quad & (a_i - x, D^{\frak a}_i) \in \LL_p^{m}, &&  \forall i \in \I, \\
    & (b - x, D^{\frak b}) \in \LL_p^{m},&& \\
    & \Lambda\rho \leq \Lambda^H D^{\frak b}, &&\\
    & \varphi^\vartheta \leq \begin{cases}
        \Lambda(\mu^\vartheta- \Lambda^\vartheta) & \text{Under DSR}\\
        \Lambda\mu^\vartheta & \text{Under ISR}
    \end{cases} && \forall \vartheta\in \Theta\\
    &(\varphi^\vartheta,\theta^\vartheta, 1) \in \SS^2_+, && \forall \vartheta\in \Theta\\
    \noalign{
    \noindent {\small Under DSR} \hdashrule[0.5ex]{.89\linewidth}{0.4pt}{1mm 1mm}
    }
    & \mu^\vartheta-\Lambda^\vartheta \geq \varepsilon, && \forall \vartheta\in\Theta\\
    & r_i\geq \kappa_1 D^{\frak a}_i+ \sum_{\vartheta\in\Theta} \Lambda^\vartheta\theta^\vartheta + \kappa_2(D^{\frak b}- \rho), & & \forall i\in\I,\\
    \noalign{
    \noindent {\small Under ISR} \hdashrule[0.5ex]{.9\linewidth}{0.4pt}{1mm 1mm}
    }
    & (\mu^\vartheta, \phi^\vartheta,1)\in \SS^2_+ &&  \forall \vartheta\in\Theta,\\
    & 1-\sum_{\vartheta\in\Theta} \Lambda^\vartheta\phi^\vartheta \geq \varepsilon, &&\\
    & \left(1-\sum_{\vartheta\in\Theta} \Lambda^\vartheta\phi^\vartheta, \varrho^\vartheta, \phi^\vartheta\right)\in \SS^2_+&& \forall \vartheta\in\Theta,\\
   & r_i\geq \kappa_1 D^{\frak a}_i+ \sum_{\vartheta\in\Theta} \Lambda^\vartheta(\theta^\vartheta+\varrho^\vartheta) + \kappa_2(D^{\frak b}- \rho), &&  \forall i\in\I,\\
   \noalign{
    \noindent {\small Auxiliary variables} \hdashrule[0.5ex]{.82\linewidth}{0.4pt}{1mm 1mm}
    }
   & \theta^\vartheta, \varphi^\vartheta,\phi^\vartheta, \varrho^\vartheta\geq 0, && \forall \vartheta\in\Theta,\\
   & D^\frak{b},\rho \geq 0 && \\
   & r_i, D_{i}^\frak{a}\geq 0, && \forall i\in\I.
\end{align*}

\begin{corollary}\label{coro:cp}
    \eqref{eq:1-cdn-cp} is an exact conic programming formulation for \eqref{1cdn}.
\end{corollary}

\begin{proof}
   Consider the following formulation for \eqref{1cdn}:
    \begin{align}
    \minimize_{x\in\R^m, \mu\in \mathcal{C}} \quad & \Phi(r_i)_{i\in\I}\nonumber\\
    \st \quad & D^{\frak a}_i \geq \D(a_i,x), \quad \forall i\in\I \label{eq:ctr1}\\
    & D^{\frak b}\geq \D(x,b)\label{eq:ctr2}\\
    & \Lambda\rho \leq \Lambda^H D^{\frak b},\nonumber\\
    & \tau\geq \E[T(\mu)]\label{eq:ctr3}\\
    & r_i\geq \kappa_1 D^{\frak a}_i+ \tau + \kappa_2 \left(D^{\frak b}-\rho\right) , \quad  \forall i\in\I,\nonumber\\
    & \begin{cases}
       \eqref{stab:dsr} & \text{Under DSR}\\
        \eqref{stab:isr} & \text{Under ISR}
    \end{cases}\nonumber\\
    & D^{\frak a}_i, r_i\geq 0, \quad \forall i\in \I\nonumber\\
    & D^{\frak b}, \tau\geq 0.\nonumber
\end{align}

The claim follows from applying Proposition \ref{prop:stab} to \eqref{stab:isr}, Proposition \ref{prop:distance} to \eqref{eq:ctr1}-\eqref{eq:ctr2}, and Theorems \ref{th:expectation_dsr} and \ref{th:expectation_isr} to \eqref{eq:ctr3} depending on the congested regime.
\end{proof}

\begin{remark}
\eqref{1cdn} can be further formulated as a complexity-optimal second-order cone program {\rm (SOCP)}~\citep{blanco2024minimal}, and can therefore be solved efficiently in polynomial time using interior-point methods~\citep[see, e.g.,][]{BoydVandenberghe2004}.
\end{remark}

Observe also that the original problem \eqref{1cdn} is only constrained in the service rates, $\mu$, through the stability condition and the bounded region $\mathcal{C}$. In case this region is induced by the efficiency-oriented bound \eqref{budget}, one can derive necessary and sufficient conditions for the feasibility of the problem above, as stated in the following result.

\begin{proposition}\label{prop:budget_feasibility}
Let $\mathcal C$ be defined as in \eqref{budget}. Then the following statements hold.

\begin{enumerate}[label=\roman*), ref=\roman*]
\item Under {\rm DSR} regime, \eqref{1cdn} is feasible if and only if $
\Gamma \;\ge\;
\Gamma_{\min}^{\mathrm{DSR}}
:=
\sum_{\vartheta\in \Theta} c^\vartheta \big(\Lambda^\vartheta+\varepsilon\big)$.

\item Under {\rm ISR} regime, \eqref{1cdn} is feasible if and only if $
\Gamma \;\ge\;
\Gamma_{\min}^{\mathrm{ISR}}
:= \frac{1}{1-\varepsilon} \left(\sum_{\vartheta\in \Theta} \sqrt{c^\vartheta \Lambda^\vartheta}\right)^2$.
\end{enumerate}
\end{proposition}

\begin{proof}
We prove necessity and sufficiency in each case.

\medskip

\begin{enumerate}[label=\roman*), ref=\roman*]
    \item Feasibility of \eqref{1cdn} requires 
    $$
\mu^\vartheta \ge \Lambda^\vartheta+\varepsilon
\; \forall \vartheta\in\Theta, \quad \text{ and }\quad
\sum_{\vartheta\in\Theta} c^\vartheta \mu^\vartheta \le \Gamma.
$$
Hence, for every feasible vector $(\mu^\vartheta)_{\vartheta\in\Theta}$, one has  $\sum_{\vartheta\in\Theta} c^\vartheta \mu^\vartheta
\;\ge\;
\sum_{\vartheta\in\Theta} c^\vartheta(\Lambda^\vartheta+\varepsilon)$. Gathering both inequalities yields the desired expression which proves necessity.

Conversely, assume that $\Gamma \ge \Gamma_{\rm min}^{\rm DSR}$, the fact $\Lambda^\vartheta+\varepsilon\in \mathcal{C}_{\rm stab}$ proves sufficiency.

\item 
Feasibility of \eqref{1cdn} requires
$$
\sum_{\vartheta\in\Theta}\frac{\Lambda^\vartheta}{\mu^\vartheta}\le 1-\varepsilon \quad \text{ and }\quad \sum_{\vartheta\in\Theta} c^\vartheta \mu^\vartheta \le \Gamma.
$$

Assume first that $(\mu^\vartheta)_{\vartheta\in\Theta}$ is feasible. Applying the Cauchy--Schwarz inequality to the vectors
$$
\left(\sqrt{c^\vartheta\mu^\vartheta}\right)_{\vartheta\in\Theta}
\qquad\text{and}\qquad
\left(\sqrt{\frac{\Lambda^\vartheta}{\mu^\vartheta}}\right)_{\vartheta\in\Theta},
$$
we obtain
$$
\left(\sum_{\vartheta\in\Theta}\sqrt{c^\vartheta\Lambda^\vartheta}\right)^2
\le
\left(\sum_{\vartheta\in\Theta} c^\vartheta\mu^\vartheta\right)
\left(\sum_{\vartheta\in\Theta}\frac{\Lambda^\vartheta}{\mu^\vartheta}\right).
$$
Using ISR stability condition, it follows that
$$
\sum_{\vartheta\in\Theta} c^\vartheta\mu^\vartheta
\ge
\frac{1}{1-\varepsilon}
\left(\sum_{\vartheta\in\Theta}\sqrt{c^\vartheta\Lambda^\vartheta}\right)^2.
$$
Since feasibility also requires $\sum_{\vartheta\in\Theta} c^\vartheta\mu^\vartheta\le\Gamma$.

Assume now that $\Gamma
\ge \Gamma_{\min}^{\rm ISR}$. Set $
\alpha:=
\frac{1}{1-\varepsilon}
\sum_{\vartheta\in\Theta}\sqrt{c^\vartheta\Lambda^\vartheta}$, 
and define
$$
\mu^\vartheta:=\alpha\sqrt{\frac{\Lambda^\vartheta}{c^\vartheta}},
\qquad \forall \vartheta\in\Theta,
$$
which are well-defined by the positivity of the costs. Then
$$
\frac{\Lambda^\vartheta}{\mu^\vartheta}
=
\frac{\sqrt{c^\vartheta\Lambda^\vartheta}}{\alpha},
\qquad \forall \vartheta\in\Theta,
$$
and summing over \(\Theta\) gives
$$
\sum_{\vartheta\in\Theta}\frac{\Lambda^\vartheta}{\mu^\vartheta}
=
\frac{1}{\alpha}\sum_{\vartheta\in\Theta}\sqrt{c^\vartheta\Lambda^\vartheta}
=
1-\varepsilon.
$$
Thus ISR stability condition holds with equality. Moreover,
$$
\sum_{\vartheta\in\Theta} c^\vartheta\mu^\vartheta
=
\alpha\sum_{\vartheta\in\Theta}\sqrt{c^\vartheta\Lambda^\vartheta}
=
\frac{1}{1-\varepsilon}
\left(\sum_{\vartheta\in\Theta}\sqrt{c^\vartheta\Lambda^\vartheta}\right)^2
\le \Gamma.
$$
Hence $(\mu^\vartheta)_{\vartheta\in\Theta}\in\mathcal C$, and therefore \eqref{1cdn} is feasible. This proves sufficiency.

\end{enumerate}
\end{proof}

In Figure~\ref{1server:ex} we illustrate the solutions obtained for the same set of demand points under the three discussed regimes. Although all formulations share a common distance-based component, they differ in the way congestion is modeled and penalized. As a consequence, the optimal location of the single edge server is influenced by the relative importance of transportation and congestion effects, which leads to different spatial trade-offs between proximity to the demand points and proximity to the origin server.

In the uncongested system (UNC), where congestion effects are negligible, the problem reduces to a purely distance-based objective. Consequently, the optimal location corresponds to a weighted geometric median (Weber-type point), and the server is positioned in a balanced location with respect to both the demand points and the origin server.

In contrast, under DSR dynamic, congestion effects are modeled separately for each traffic class, and the component associated with misses induces a stronger dependence on the origin server. This results in an optimal solution in which the edge server is located closer to the origin server in order to reduce the impact of miss-related costs.

Finally, under ISR dynamic, congestion is modeled through a shared service mechanism that couples all traffic classes. This leads to a more balanced trade-off between distance and congestion, and the resulting location lies between the demand region and the origin server, reflecting the joint influence of both components.

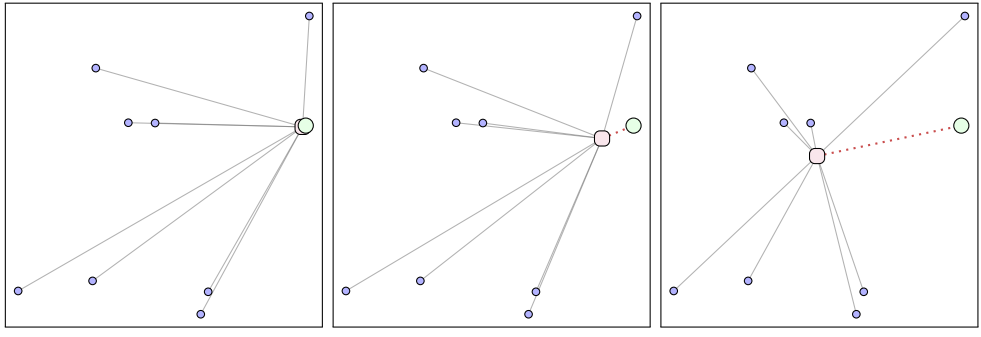
\begin{figure}
    \centering
    \fbox{\begin{tikzpicture}[
    >=Latex,
    scale=0.85,
    demand/.style    = {draw,fill=blue!30,circle,minimum size=1.0mm,inner sep=0pt},
    originbox/.style = {draw,circle,fill=green!10,minimum size=2mm,inner sep=0pt},
    facility/.style  = {draw,font=\tiny,fill=purple!10,rounded corners=2pt,minimum width=2.0mm,minimum height=2.0mm,inner sep=0pt},
    demandlink/.style = {thin,black!55,opacity=0.55},
    originlink/.style = {line width=0.8pt,dotted,red!70!black,opacity=0.7},
    every node/.style = {align=center}
]
  \draw[demandlink] (-2.24592,-1.53783) -- (0.46535,-1.60264);
  \draw[demandlink] (-2.80217,-3.99732) -- (0.46535,-1.60264);
  \draw[demandlink] (-1.82955,-1.54249) -- (0.46535,-1.60264);
  \draw[demandlink] (-2.75207,-0.68859) -- (0.46535,-1.60264);
  \draw[demandlink] (-1.00381,-4.16665) -- (0.46535,-1.60264);
  \draw[demandlink] (-3.95998,-4.15390) -- (0.46535,-1.60264);
  \draw[demandlink] (0.56995,0.12194) -- (0.46535,-1.60264);
  \draw[demandlink] (-1.11891,-4.51576) -- (0.46535,-1.60264);
  \draw[originlink] (0.46535,-1.60264) -- (0.51456,-1.58256);
  \node[demand] at (-2.24592,-1.53783) {};
  \node[demand] at (-2.80217,-3.99732) {};
  \node[demand] at (-1.82955,-1.54249) {};
  \node[demand] at (-2.75207,-0.68859) {};
  \node[demand] at (-1.00381,-4.16665) {};
  \node[demand] at (-3.95998,-4.15390) {};
  \node[demand] at (0.56995,0.12194) {};
  \node[demand] at (-1.11891,-4.51576) {};
    \node[facility] at (0.46535,-1.60264) {};
  \node[originbox] at (0.51456,-1.58256) {};
\end{tikzpicture}}~\fbox{\begin{tikzpicture}[
    >=Latex,
    scale=0.85,
    demand/.style    = {draw,fill=blue!30,circle,minimum size=1.0mm,inner sep=0pt},
    originbox/.style = {draw,circle,fill=green!10,minimum size=2mm,inner sep=0pt},
    facility/.style  = {draw,font=\tiny,fill=purple!10,rounded corners=2pt,minimum width=2.0mm,minimum height=2.0mm,inner sep=0pt},
    demandlink/.style = {thin,black!55,opacity=0.55},
    originlink/.style = {line width=0.8pt,dotted,red!70!black,opacity=0.7},
    every node/.style = {align=center}
]
  \draw[demandlink] (-2.24592,-1.53783) -- (0.02435,-1.78233);
  \draw[demandlink] (-2.80217,-3.99732) -- (0.02435,-1.78233);
  \draw[demandlink] (-1.82955,-1.54249) -- (0.02435,-1.78233);
  \draw[demandlink] (-2.75207,-0.68859) -- (0.02435,-1.78233);
  \draw[demandlink] (-1.00381,-4.16665) -- (0.02435,-1.78233);
  \draw[demandlink] (-3.95998,-4.15390) -- (0.02435,-1.78233);
  \draw[demandlink] (0.56995,0.12194) -- (0.02435,-1.78233);
  \draw[demandlink] (-1.11891,-4.51576) -- (0.02435,-1.78233);
  \draw[originlink] (0.02435,-1.78233) -- (0.51456,-1.58256);
  \node[demand] at (-2.24592,-1.53783) {};
  \node[demand] at (-2.80217,-3.99732) {};
  \node[demand] at (-1.82955,-1.54249) {};
  \node[demand] at (-2.75207,-0.68859) {};
  \node[demand] at (-1.00381,-4.16665) {};
  \node[demand] at (-3.95998,-4.15390) {};
  \node[demand] at (0.56995,0.12194) {};
  \node[demand] at (-1.11891,-4.51576) {};
  \node[originbox] at (0.51456,-1.58256) {};
  \node[facility] at (0.02435,-1.78233) {};
\end{tikzpicture}}~\fbox{\begin{tikzpicture}[
    >=Latex,
    scale=0.85,
    demand/.style    = {draw,fill=blue!30,circle,minimum size=1.0mm,inner sep=0pt},
    originbox/.style = {draw,circle,fill=green!10,minimum size=2mm,inner sep=0pt},
    facility/.style  = {draw,font=\tiny,fill=purple!10,rounded corners=2pt,minimum width=2.0mm,minimum height=2.0mm,inner sep=0pt},
    demandlink/.style = {thin,black!55,opacity=0.55},
    originlink/.style = {line width=0.8pt,dotted,red!70!black,opacity=0.7},
    every node/.style = {align=center}
]
  \draw[demandlink] (-2.24592,-1.53783) -- (-1.72963,-2.05581);
  \draw[demandlink] (-2.80217,-3.99732) -- (-1.72963,-2.05581);
  \draw[demandlink] (-1.82955,-1.54249) -- (-1.72963,-2.05581);
  \draw[demandlink] (-2.75207,-0.68859) -- (-1.72963,-2.05581);
  \draw[demandlink] (-1.00381,-4.16665) -- (-1.72963,-2.05581);
  \draw[demandlink] (-3.95998,-4.15390) -- (-1.72963,-2.05581);
  \draw[demandlink] (0.56995,0.12194) -- (-1.72963,-2.05581);
  \draw[demandlink] (-1.11891,-4.51576) -- (-1.72963,-2.05581);
  \draw[originlink] (-1.72963,-2.05581) -- (0.51456,-1.58256);
  \node[demand] at (-2.24592,-1.53783) {};
  \node[demand] at (-2.80217,-3.99732) {};
  \node[demand] at (-1.82955,-1.54249) {};
  \node[demand] at (-2.75207,-0.68859) {};
  \node[demand] at (-1.00381,-4.16665) {};
  \node[demand] at (-3.95998,-4.15390) {};
  \node[demand] at (0.56995,0.12194) {};
  \node[demand] at (-1.11891,-4.51576) {};
  \node[originbox] at (0.51456,-1.58256) {};
  \node[facility] at (-1.72963,-2.05581) {};
\end{tikzpicture}}
    \caption{Solutions under the three different regimes (from left to right: DSR, ISR, and UNC).\label{1server:ex}}
\end{figure}

\section{General Discrete Conic Optimization Framework} \label{sec:conicOPT}

This section is due to present the whole optimization framework used to design the hierarchical service architecture under a general framework for the CDN with several edge and origin servers. In this general case, the model should jointly determine the placement of a given number, $n$, of edge servers, the assignment of demand points to edge servers, the routing of edge servers to origin servers, and the different service capacities installed at each edge node. 

The decision variables in the general case describe three fundamental aspects of the system design: 
(i) the allocation of traffic between network components, 
(ii) the physical placement of edge infrastructure, and 
(iii) the processing capacity installed at each service node.

\subsection*{Design Variables}

We introduce the variables that characterize the structure and operation of the service system. They describe how demand is routed through the network, where edge servers are located in the physical space, and the service capacities installed at each edge server for handling different types of requests.\\

\noindent {\bf Placement of edge servers.} Edge servers can be located anywhere in a continuous geographical region. Their spatial positions influence the latency experienced by users and therefore affect the assignment of demand points. For each $j \in \J$:
$$
x_j \in \mathbb{R}^m: 
\;
\text{Location coordinates of edge server $j$ in the $m$-dimensional space}.
$$
\noindent {\bf Assignment of demand points to edge servers.} Each demand point represents a geographical location generating user requests (e.g., a population centroid or a traffic aggregation node). Binary variables indicate which edge server is responsible for processing the requests generated at each demand point.
\begin{equation}\label{eq:z-alloc}
z_{ij} =
\begin{cases}
1 & \text{if demand point $i$ is served by edge server $j$},\\
0 & \text{otherwise},
\end{cases}
\qquad \forall i \in \I,\; j \in \J.
\end{equation}
These variables determine how user traffic is distributed across the edge infrastructure.\\

\noindent {\bf Assignment of edge servers to origin servers.} When a request of type $\vartheta \in \Theta\backslash\{H\}$ cannot be processed locally at edge server $j$, it must be forwarded to an origin server $k$ providing the corresponding service capacity. To model this, we introduce the following binary variables:
\begin{equation}\label{eq:y-alloc}
y_{jk} =
\begin{cases}
1 & \text{if edge server $j$ is assigned to origin server $k$},\\
0 & \text{otherwise},
\end{cases}
\qquad \forall j \in \J,\; k \in \K.
\end{equation}
These variables define the upstream routing decisions, thereby linking edge servers to the service capacities available at the origin servers.\\

\noindent {\bf Service capacities at edge servers.} Each edge server processes requests that may correspond to cache hits or cache misses. Depending on the adopted queueing architecture, these streams can either be served independently (DSR) or jointly through a shared service mechanism (ISR).

To capture both situations in a unified framework, we define the service capacities at each edge server $j \in \J$ as
$$
\mu_j^\vartheta \ge 0, \qquad \forall \vartheta \in \Theta.
$$

Under DSR, the different types of cache are processed in independent service subsystems, each with its own capacity $\mu_j^\vartheta$, for all $\vartheta \in \Theta$.

Under ISR, both types of requests are processed within a single shared queue. In this case, $\{\mu_j^\vartheta\}_{\vartheta \in \Theta}$  represents effective service parameters that characterize the service-time mixture rather than physically separated capacities.

Thus, these variables represent the processing capability installed at each edge node under a general modeling perspective that encompasses both regimes.

We assume that the service capacities belong to a bounded region $\mathcal{C} \subseteq \R^{|\Theta|\times |\J|}$ which, analogously to \eqref{1cdn}, can represent a budget constraint. In particular, we consider
$$
\mathcal{C} = \left\{(\mu_j^\vartheta)_{j\in \J,\, \vartheta\in \Theta} \in \R^{|\Theta|\times |\J|}_+ :
\sum_{j\in \J} \sum_{\vartheta\in \Theta} c_j^\vartheta \mu_j^\vartheta \leq \Gamma \right\},
$$
where $c_j^\vartheta >0$ denotes the unit cost associated with the service capacity of type $\vartheta$ at edge server $j$, and $\Gamma>0$ represents the total budget available for capacity installation.

\subsection*{Constraints}

The constraints of the model describe how the system operates and ensure that the resulting design is both structurally consistent and operationally feasible. From an engineering standpoint, these constraints enforce three key requirements. First, user traffic must be routed through the network in a consistent hierarchical manner. Second, the processing capacity installed at each edge server must be sufficient to guarantee stable queueing operation. Third, the geometric distances that determine latency and routing costs must be computed consistently with the placement of the infrastructure.

In order to ensure assumption \ref{a1}, we need to impose single allocation constraints between demand points and edge servers, and between the latter and origin servers.\\

\noindent{\bf Single allocations.} The system follows a hierarchical architecture. Each demand point sends its requests to exactly one edge server, and each edge server forwards cache-miss requests to exactly one origin server. This reflects a typical operational configuration in which routing policies avoid traffic splitting in order to simplify system management and ensure predictable service behavior.

These requirements are enforced through the following assignment constraints:

\begin{align}
    \sum_{j \in \J} z_{ij} = 1, 
    & \qquad \forall i \in \I, \label{assign:edge}\tag{assign$^{\frak a}$}\\
    \sum_{k \in \K} y_{jk} = 1, 
    & \qquad \forall j \in \J.\label{assign:origin}\tag{assign$^{\frak b}$}
\end{align}

The first set of constraints guarantees that the traffic generated at each demand point $i$ is entirely handled by a single edge server. The second set ensures that every edge server is connected to exactly one origin server for processing cache-miss requests.

\begin{theorem}\label{th:general_cdn}
    The following mixed-integer nonlinear programming {\rm (MINLP)} model is an exact formulation for the content delivery network design problem under uncongested, {\rm DSR}, and {\rm ISR} dynamics,

\begin{align*}\label{eq:cnd}\tag{\rm CDN-MICP}
    \minimize_{x\in\R^{n\times m}, \; \mu \in \mathcal{C}} \quad & \Phi(r_i)_{i\in\I} && \\
    \st \quad & \eqref{assign:edge}, \eqref{assign:origin}, &&\\
    & ((a_i - x_j)\highlight{z_{ij}}, D^{\frak a}_i) \in \LL_p^{m}, &&  \forall i \in \I,j\in\J, \\
    & ((b_k - x_j)\highlight{y_{jk}}, D_j^{\frak b}) \in \LL_p^{m}, && \forall j \in \J,k\in\K,\\
    & \sum_{i\in\I}\xi_i\highlight{z_{ij}\rho_j} \leq  \sum_{i\in\I}p_i^H\xi_i\highlight{z_{ij}D_j^{\frak b}}, && \forall j\in \J,\\
    & \varphi^\vartheta_j \leq \begin{cases}
         \displaystyle\sum_{i\in\I}\xi_i\highlight{z_{ij}\mu_j^\vartheta}-  \displaystyle\sum_{i,i'\in\I}p_i^\vartheta\xi_i\xi_{i'}\highlight{z_{ij}z_{i'j}} & \text{Under {\rm DSR}}\\
         \displaystyle\sum_{i\in\I}\xi_i\highlight{z_{ij}\mu_j^\vartheta} & \text{Under {\rm ISR}}
    \end{cases} && \forall \vartheta\in \Theta, j\in\J,\\
    &(\varphi_j^\vartheta,\theta_j^\vartheta, 1) \in \SS^2_+, && \forall \vartheta\in \Theta,j\in \J,\\
    \noalign{
    \noindent {\small Under {\rm DSR}} \hdashrule[0.5ex]{.89\linewidth}{0.4pt}{1mm 1mm}
    }
    & \mu_j^\vartheta-\sum_{i\in\I}p_i^\vartheta\xi_iz_{ij} \geq \varepsilon, && \forall \vartheta\in\Theta, j\in \J\\
    & r_i\geq \kappa_1 D^{\frak a}_i+ \left[\sum_{\vartheta\in\Theta} \sum_{i\in\I}p_i^\vartheta\xi_i\highlight{z_{ij}\theta_j^\vartheta} + \kappa_2(D^{\frak b}_j-\rho_j)\right]\highlight{z_{ij}}, & & \forall i\in\I,j\in \J,\\
    \noalign{
    \noindent {\small Under {\rm ISR}} \hdashrule[0.5ex]{.9\linewidth}{0.4pt}{1mm 1mm}
    }
    & (\mu_j^\vartheta, \phi_j^\vartheta,1)\in \SS^2_+ &&  \forall \vartheta\in\Theta,j\in \J,\\
    & 1-\sum_{\vartheta\in\Theta} \sum_{i\in\I}p_i^\vartheta\xi_i\highlight{z_{ij}\phi_j^\vartheta} \geq \varepsilon, && \forall j\in\J,\\
    & \left(1-\sum_{\vartheta\in\Theta}\sum_{i\in\I}p_i^\vartheta\xi_i\highlight{z_{ij}\phi_j^\vartheta}, \varrho_j^\vartheta, \phi_j^\vartheta\right)\in \SS^2_+&& \forall \vartheta\in\Theta,j\in \J,\\
   & r_i\geq \kappa_1 D^{\frak a}_i+ \left[\sum_{\vartheta\in\Theta} \sum_{i\in\I}p_i^\vartheta\xi_i\highlight{z_{ij}(\theta_j^\vartheta+\varrho_j^\vartheta)} + \kappa_2(D^{\frak b}_j-\rho_j)\right]\highlight{z_{ij}}, &&  \forall i\in\I,j\in \J,\\
   \noalign{
    \noindent {\small Binary decisions} \hdashrule[0.5ex]{.83\linewidth}{0.4pt}{1mm 1mm}
    }
    & z_{ij} \in \{0,1\}, && \forall i \in \I, j \in \J,\\
   & y_{jk} \in \{0,1\}, && \forall j \in \J, k \in \K,\\
   \noalign{
    \noindent {\small Auxiliary variables} \hdashrule[0.5ex]{.83\linewidth}{0.4pt}{1mm 1mm}
    }
   & \theta_j^\vartheta, \varphi_j^\vartheta,\phi_j^\vartheta, \varrho_j^\vartheta\geq 0, && \forall \vartheta\in\Theta,j\in \J,\\
   & D_j^\frak{b},\rho_j \geq 0 && \forall j\in\J, \\
   & r_i, D_{i}^\frak{a}\geq 0, && \forall i\in\I.
\end{align*}
\end{theorem}

\begin{proof}
Given allocation variables $z$ \eqref{eq:z-alloc}, the total request arrival rate handled by each edge server $j$ and its decomposition into hit and miss streams are given by
$$
\Lambda_j(z): = \sum_{i \in \I} \xi_i z_{ij}, \quad \text{and}\quad \Lambda_j^\vartheta(z) := \sum_{i \in \I} p_i^\vartheta \xi_i z_{ij}, \; \forall \vartheta\in \Theta.
$$
Here, $\xi_i$ denotes the demand rate generated at point $i$, and $p_i^\vartheta$ is the probability that a request generated at $i$ results in a cache of type $\vartheta$ at the edge server. Consequently, $\Lambda_j(z)$ represents the total incoming traffic to edge server $j$, while $\Lambda_j^\vartheta(z)$  corresponds to the cache of type $\vartheta$ arrival rates. Besides, we can define $\I_j(z):=\{i\in \I : z_{ij}=1\}$ as the set of demand points assigned to server $j \in \J$ by $z$. Together with allocation variables $y$ \eqref{eq:y-alloc}, we can consider for each edge server $j\in\J$ the problem {\rm CDN}$_j(z,y)$ as the \eqref{1cdn} defined by the input data: $\{a_i : i\in \I_j(z)\}$ as demand points, $b_k$ with $y_{jk}=1$ as the origin server, and $\Lambda_j(z),\Lambda_j^\vartheta(z)$ for all $\vartheta\in\Theta$ as the arrival flows. Then, the global CDN design problem can be seen as 

\begin{align*}
    \minimize \; & \sum_{j\in\J} {\rm CDN}_j(z,y)\\
    \st \; & \mu \in \mathcal{C}\\
    & \eqref{assign:edge}, \eqref{assign:origin},\\
    & z_{ij} \in \{0,1\}, \quad \forall i \in \I, j \in \J,\\
   & y_{jk} \in \{0,1\}, \quad \forall j \in \J, k \in \K.
\end{align*}

\eqref{eq:cnd} is a reformulation of the above problem by Corollary \ref{coro:cp}. Hence, it is an exact formulation for the CDN design problem as claimed.

\end{proof}
To facilitate the identification of the differences between the conic optimization problem derived in the previous section for the single-server case and the multi-server formulation, which requires the binary assignment variables $z$ and $y$, in model~\eqref{eq:cnd} we have highlighted in color the bilinear terms that arise as a consequence of the multi-server setting. Note that this {\rm MINLP} problem admits a reformulation as a mixed-integer conic program (MICP) induced by the McCormick-type envelopes (see Appendix~\ref{app_McCormick}) of the products by binary variables highlighted in the formulation. In general, we have to represent expressions $w=f(v)z$ with finite bounds $0\leq f(v)\leq UB$ for all decisions $v$ at optimality. These upper bounds depend on the nature of the images of $f$ and could be:
\begin{itemize}
    \item[--] induced by the diameter of the spatial breadth of the system, $$UB=\Delta:= \max(\{\|a_i-a_{i'}\|_p : i,i'\in \I\}\cup \{\|b_k-b_{k'}\|_p : k,k'\in \K\});$$ 

    \item[--] induced by the budget limitation of the service capacity, $UB=\frac{\Gamma}{c^\vartheta_j}$;

    \item[--] induced by a sufficiently large threshold on the response time, $UB=\T$;

    \item[--] or $UB=1$ in case $f(v)$ is a binary decision as well.
\end{itemize}

The McCormick envelopes to convexify \eqref{eq:cnd} fit into one of these types:
\begin{itemize}
    \item[--] $f$ is the identity in $\R_+$, leading to an expressions kind: $w=vz$. Namely,
    \begin{itemize}
        \item[$\bullet$] If $v\in\{\mu,\rho\}$, since the objective naturally drives $v$ to increase, then the envelope becomes
        \begin{align*}
            w&\leq UBz\\
            w&\leq v,
        \end{align*}
        where $UB=\frac{\Gamma}{c^\vartheta_j}$ for $v=\mu$, and $UB=\Delta$ for $v=\rho$.
        \item[$\bullet$] If $v\in \{D^{\frak b},\theta,\phi,\varrho\}$, since the objective naturally drives $v$ to decrease, then the envelope becomes
        \begin{equation*}
            v - UB(1-z)  \le w,
        \end{equation*}
        where $UB=\Delta$ for $v=D^{\frak b}$, and $UB=\T$ for $v=\theta,\phi,\varrho$.
        
        \item[$\bullet$] If $v=z'$ (stands for others unnecessary equality-to-$z$ allocation variables, so $UB=1$). Since the objective naturally drives the allocations to be zero, then the envelope becomes
        \begin{equation*}
            z + z' -1 \le w.
        \end{equation*}
    \end{itemize}

    \item[--] $f(v)=\|a-v\|_p:\R^m \to \R$. Since the objective naturally drives $f(v)$ to decrease, then the envelope becomes
    \begin{equation*}
        \|a-v\|_p  \le w + \Delta(1-z).
    \end{equation*}

    \item[--] $f(v_0,v^\vartheta_1,\ldots,v^\vartheta_{|\I|})= \sum_{\vartheta\in\Theta} \sum_{i\in\I}p_i^\vartheta\xi_iv^\vartheta_i + \kappa_2v_0 : \R_+\times \R_+^{|\Theta||\I|} \to \R$. Since the objective naturally drives $f(v_0,v^\vartheta_1,\ldots,v^\vartheta_{|\I|})$ to decrease, then the envelope becomes
    \begin{equation*}
        \sum_{\vartheta\in\Theta} \sum_{i\in\I}p_i^\vartheta\xi_iv^\vartheta_i + \kappa_2v_0 - \T(1-z) \le w.
    \end{equation*}
    \end{itemize}

\begin{theorem}\label{th:complexity}
    The mixed-integer cone program \eqref{eq:cnd} is {\rm NP}-hard.
\end{theorem}

\begin{proof}
    Consider the simpler setting of a planar ($m=2$) and Euclidean ($p=2$) CDN operating under no congestion and with no delay between edges and origin servers ($\kappa_2=0$). 
    The problem reduces to a \textit{classical} continuous multifacility location problem~\citep{blanco2014revisiting}, whose NP-hardness  follows from the NP-hardness of its simpler discrete versions, such as Euclidean $p$-median ($\Phi=\sum$) and $p$-center ($\Phi=\max$) problems \citep[see, e.g.,][]{megiddo1984complexity}.
\end{proof}

\begin{theorem}\label{th:single_edge_continuous}
Assume that $|\J|=1$. Then, the binary requirements $y_k\in\{0,1\}$, $k\in\K$,
can be replaced by their continuous relaxation $0\leq y_k\leq 1$ without changing the optimal value. Moreover, the relaxed problem admits an optimal solution in which $y$ is binary. Thus, this case of {\rm CDN} design problem reduces to a continuous conic optimization problem under the uncongested, {\rm DSR}, and {\rm ISR} dynamics.
\end{theorem}
\begin{proof}
The result follows from the observation that in case $|\J|=1$, the only remaining discrete decisions are the origin-assignment variables $y_k$. Thus, for fixed values of all variables except $y$, the feasible set of $y$ is the $|\K|$-dimensional simplex. Since the objective function $\Phi$ is assumed to be nondecreasing in the response-time vector, an optimal solution minimizes this origin-distance contribution over the simplex. Therefore, for fixed edge server location, the relaxed origin-assignment subproblem is a linear minimization problem over a simplex, and then, it always admits an optimal solution at an extreme point of the simplex.
\end{proof}

Once a feasible infrastructure design is determined, the resulting system must be evaluated according to performance criteria that quantify the quality of service experienced by the users. This choice is represented through the aggregation function $\Phi$ applied to the vector of expected response times encoded by $r=(r_i)_{i\in\I}$.

In this work, we restrict attention to aggregation functions $\Phi:\mathbb{R}^{|\I|}\to\mathbb{R}$ that are conic representable, nondecreasing in each component, and lower semicontinuous.

Table~\ref{tab:objectives} describes the three criteria used in the computational experiments.
\renewcommand{\arraystretch}{1.2}
\begin{table}[htbp]
\centering
\begin{tabular}{m{3.6cm} >{\centering\arraybackslash}m{2.2cm} m{5.6cm}}
\toprule
\textbf{Objective} & $\Phi(r)$ & \textbf{Criterion} \\
\midrule
{\bf SUM} of response times&
$\dsum_{i\in\I} r_i$
&
Minimizes the total expected response time and captures overall system efficiency.\\
{\bf C}onditional {\bf V}alue-{\bf a}t-{\bf R}isk &
${\rm CVaR}_\alpha(r_i)_{i\in\I}$
&
Emphasizes tail performance by averaging the largest fraction
$(1-\alpha)$ of expected response times.\\
{\bf EXP}onential penalty &
$\dsum_{i\in\I} e^{\zeta r_i},
\; \zeta>0$
&
Strongly penalizes large delays, promoting balanced solutions.\\
\bottomrule
\end{tabular}
\caption{Different objective functions for response-time minimization. Acronyms for the different criteria are highlighted in bold face.}
\label{tab:objectives}
\end{table}

Beyond these objectives, several alternative aggregation functions can be considered within the same framework, provided they satisfy the above structural properties such as maximum, ordered weighted averages, H\"older norms, or threshold penalties.

\section{Computational Experiments}
\label{sec:computational}

This section evaluates the computational performance, scalability, and structural behavior of the proposed mixed-integer conic formulations on instances derived from real Internet topology data~\citep{ITDK}. The purpose of the experiments is twofold. From a computational viewpoint, we assess whether the conic liftings developed in the previous sections lead to formulations that can be handled reliably by off-the-shelf conic optimization solvers. From a modeling viewpoint, we quantify the effect of explicitly incorporating queueing congestion into CDN design decisions.

The experiments pursue three complementary objectives. First, we assess the tractability of the proposed conic reformulations under different service regimes and objective functions. Second, we analyze the managerial implications of congestion-aware modeling on CDN design decisions and service quality. Third, we investigate how the interaction between latency, cache effectiveness, and queueing dynamics affects the resulting network configurations.

The benchmark instances are constructed from the CAIDA MIDAR Internet topology dataset (see Appendix~\ref{app:instances}), thereby providing geographically meaningful and structurally heterogeneous CDN scenarios. The experimental framework allows systematic variation of the demand size, infrastructure scale, congestion parameters, and budget availability while preserving realistic topological characteristics. Overall, the experiments are designed to evaluate not only the computational viability of the proposed formulations but also the operational insights obtained from explicitly integrating queueing effects into CDN optimization models.

\subsection{Experimental Design}

The experimental design systematically varies the principal structural and operational parameters of the CDN problem. The number of demand points $|\I|$, edge servers $|\J|$, and origin servers $|\K|$ are varied to analyze scalability and assignment complexity. The cardinality of the edge layer is selected as a function of the demand size so as to preserve realistic infrastructure sparsity while progressively increasing routing flexibility.

The parameter $\kappa_2$ controls the relative importance of the edge-to-origin propagation delay and therefore modulates the impact of cache misses on service quality. In addition, the total service budget is controlled through a scaling factor $\beta$ multiplying the minimum feasible budget $\Gamma_{\min}$, guaranteeing queue stability.  Since the stability budget is computed so as to ensure feasibility under the most restrictive queueing regime, it provides a common reference level for comparing DSR, ISR, and UNC dynamics. This allows us to investigate operational regimes ranging from nearly saturated systems to moderately over-provisioned infrastructures.

To evaluate the impact of different service-quality criteria, we consider the three objective functions described in Table~\ref{tab:objectives}, with confidence level $\alpha=0.90$ for the CVaR, and $\zeta=0.005$ for the exponential penalty. These objectives allow us to analyze different trade-offs between efficiency, fairness, and robustness, respectively.

Table~\ref{tab:parameters} gathers the experimental parameters.

\renewcommand{\arraystretch}{1}
\begin{table}[htbp]
\centering
\begin{tabular}{lll}
\toprule
{\bf Parameter} & {\bf Description} & {\bf Values} \\
\midrule
$|\I|$ & Number of demand points &
$\{10,20,30,40,50,60,70,80,90,100,150, 200\}$ \\
$|\J|$ & Number of edge servers &
$
\left\{
\begin{array}{ll}
\{1,2\}, & |\I|=10,\\
\{1,2,3\}, & |\I|\geq 20
\end{array}
\right.
$\\
$|\K|$ & Number of origin servers & $\{1,3,5\}$\\
$\kappa_1$ & User-to-edge latency weight & $1.0$\\
$\kappa_2$ & Edge-to-origin latency weight & $\{0.05,0.5,1.5\}$\\
$\beta$ & Budget scaling factor & $\{1.01,1.10\}$\\
$\Phi$ & Objective function &
$\{\mathrm{SUM},\mathrm{CVaR},\mathrm{EXP}\}$\\
\bottomrule
\end{tabular}
\caption{Experimental parameters.}
\label{tab:parameters}
\end{table}

All experiments were conducted on a Mac Studio (2022) equipped with Apple Silicon architecture and $64$GB unified memory, running macOS 26.3.1. The formulations were implemented in Python and solved using MOSEK~11.1 under a single-threaded environment. A time limit of 2 hours was set for all instances. All instances and results are publicly available in our GitHub repository~\citep{CDN2026}.

To characterize the structural properties of the generated benchmark instances, Table~\ref{tab:data_statistics} reports aggregate statistics associated with the topology, demand intensities, cache-hit probabilities, and spatial dispersion of the sampled networks. In particular, $\overline{\deg}$ denotes the average node degree of the sampled demand graph and measures the local connectivity of the underlying network structure. The quantity $\mathrm{CV}(\xi)$ represents the coefficient of variation of the demand intensities $\xi_i$, capturing the relative heterogeneity of the traffic generated across demand points, whereas $G(\xi)$ corresponds to the associated Gini index and quantifies the degree of inequality in the demand distribution. The statistics $\overline{p_i}$ and $\sigma(p_i)$ denote, respectively, the mean and standard deviation of the cache-hit probabilities assigned to the demand points, thereby describing both the average cache effectiveness and its variability throughout the network. The spatial descriptor $\overline{\D}$ represents the average pairwise Euclidean distance between demand locations, while $\overline{\Delta}$ corresponds to the average geometric diameter of the instances, that is, the maximum pairwise distance between demand points. The first four columns summarize the global distribution of these descriptors across all generated configurations, whereas the remaining columns analyze their evolution as the number of demand points increases. The results show persistent heterogeneity in both demand intensities and cache effectiveness, as reflected by the nonzero coefficients of variation and Gini indices across all configurations. At the same time, these statistics remain remarkably stable as the problem dimension grows: the average degree varies only moderately, the traffic heterogeneity indicators remain within relatively narrow intervals, and the average cache-hit probabilities stay concentrated around $0.55$--$0.56$. The spatial descriptors further reveal geographically dispersed instances with large diameters and relatively stable average pairwise distances, indicating that the generation procedure preserves coherent geographic coverage together with nontrivial propagation-delay structures. Overall, the reported statistics confirm that the proposed benchmark generation framework produces structurally heterogeneous, geographically meaningful, and statistically stable CDN instances across different scales.

\begin{table}[htbp]
\centering
\adjustbox{scale=0.85}{
\begin{tabular}{lrrrrrrrrrrrr}
\toprule
&
\multicolumn{4}{c}{Global statistics}
&
\multicolumn{8}{c}{By instance size $|\I|$}
\\
\cmidrule(lr){2-5}
\cmidrule(lr){6-13}
Statistic
& Mean
& Std
& Min
& Max
& 10
& 20
& 30
& 50
& 80
& 100
& 150
& 200
\\
\midrule
$\overline{\deg}$
& 1.193 & 0.204 & 1.000 & 2.100
& 1.06 & 1.24 & 1.31 & 1.24 & 1.17 & 1.15 & 1.15 & 1.14
\\
$\mathrm{CV}(\xi)$
& 0.288 & 0.141 & 0.117 & 0.611
& 0.14 & 0.23 & 0.35 & 0.33 & 0.29 & 0.27 & 0.28 & 0.28
\\
$G(\xi)$
& 0.109 & 0.033 & 0.062 & 0.220
& 0.07 & 0.10 & 0.13 & 0.12 & 0.11 & 0.10 & 0.10 & 0.10
\\
$\overline{p_i}$
& 0.558 & 0.008 & 0.550 & 0.594
& 0.552 & 0.560 & 0.562 & 0.560 & 0.557 & 0.556 & 0.556 & 0.556
\\
$\sigma(p_i)$
& 0.036 & 0.030 & 0.000 & 0.116
& 0.006 & 0.026 & 0.049 & 0.043 & 0.036 & 0.033 & 0.035 & 0.037
\\
$\overline{\D}$
& 93.714 & 5.747 & 75.676 & 115.437
& 100.3 & 100.1 & 94.4 & 91.8 & 92.3 & 92.1 & 91.7 & 92.4
\\
$\overline{\Delta}$
& 283.646 & 28.138 & 194.687 & 324.916
& 250.4 & 267.2 & 279.1 & 281.0 & 293.4 & 300.7 & 308.3 & 314.3
\\
\bottomrule
\end{tabular}}
\caption{Structural statistics of the generated benchmark instances.}
\label{tab:data_statistics}
\end{table}

\subsection{Single-Edge CDN}

We first analyze the particular case in which a single edge server is deployed, namely $|\J|=1$. By Theorem~\ref{th:single_edge_continuous}, the resulting model is a conic optimization problem. This remains true even when several origin servers are available, since the model contains continuous location and capacity decisions but no demand-to-edge assignment decisions. Hence, this setting isolates the intrinsic numerical behavior of the conic queueing reformulations from the additional combinatorial complexity that appears in the multi-edge case. Although structurally simpler than the general multi-edge setting, this regime already captures the fundamental interaction between geographical latency, cache performance, congestion effects, and service-capacity decisions.

\subsubsection*{Computational Performance}

The obtained results indicate that all three formulations remain highly tractable across the entire benchmark set. Table~\ref{tab:single_edge_scalability_I} reports the scalability of the formulations with respect to the number of demand points. Even for the largest tested instances with $|\I|=200$, the median solution time remains below three seconds, while the average optimality gaps are zero throughout all experiments. The only discrete decisions in this setting are associated with the selection of the origin server, and they do not introduce a significant branch-and-bound burden. These results confirm that the proposed conic reformulations yield tight and numerically stable formulations in the single-edge setting.

\begin{table}
\begin{tabular}{r>{\centering\arraybackslash}p{2.5cm}>{\centering\arraybackslash}p{2.5cm}>{\centering\arraybackslash}p{2.5cm}}
\toprule
$|\I|$ & CPU Time (median, in secs.) &  \# IPM Iterations (mean) & \# Relaxations (mean) \\
\midrule
10 & 0.01 &  141.00 & 6.40 \\
20 & 0.02 &  149.00 & 6.40 \\
30 & 0.03 &  155.70 & 6.40 \\
40 & 0.03 &  159.70 & 6.40 \\
50 & 0.07 &  151.40 & 6.50 \\
60 & 0.11 &  154.00 & 6.60 \\
70 & 0.13 &  158.70 & 6.60 \\
80 & 0.10 &  187.00 & 6.60 \\
90 & 0.11 &  193.70 & 6.60 \\
100 & 0.67 &  197.20 & 6.60 \\
150 & 1.32 &  209.70 & 6.60 \\
200 & 2.15 &  228.10 & 6.50 \\
\bottomrule
\end{tabular}
\caption{Scalability of the single-edge formulations with respect to the number of demand points.}
\label{tab:single_edge_scalability_I}
\end{table}

The numerical behavior of the formulations is further summarized in Table~\ref{tab:single_edge_numerical_strength}. Several observations emerge. First, the uncongested approximation (UNC) consistently provides the smallest computational burden, as expected, since it avoids the nonlinear queueing interactions present in the congestion-aware formulations. However, the computational overhead introduced by the queueing dynamics remains moderate. In particular, the average solution times remain below one second for both DSR and ISR over the entire benchmark set.


Second, the DSR formulation systematically outperforms ISR in terms of running times and interior-point method (IPM) iterations. Averaged over all instances, DSR requires approximately $173$ interior-point iterations, whereas ISR requires around $210$. This behavior is consistent across all objective functions and reflects the stronger nonlinear coupling introduced by the integrated ${\rm M/H}_{|\Theta|}/1$ queueing representation of ISR. In contrast, DSR decomposes congestion effects into independent queueing subsystems, which leads to numerically milder conic interactions.

Third, the number of interior-point iterations grows only moderately with the problem size. As shown in Table~\ref{tab:single_edge_scalability_I}, the average number of barrier iterations increases from approximately $140$ for small instances to around $230$ for the largest instances considered. Similarly, the number of conic relaxations explored by the branch-and-bound procedure remains remarkably stable across all scales. Together with the zero optimality gaps, these results suggest that the proposed conic liftings generate strong and numerically stable formulations whose continuous relaxations remain highly informative even in the presence of nonlinear queueing effects.

\begin{table}
\begin{tabular}{ll>{\centering\arraybackslash}p{2.3cm}>{\centering\arraybackslash}p{2.5cm}>{\centering\arraybackslash}p{2.5cm}>{\centering\arraybackslash}p{1.9cm}}
\toprule
Regime & $\Phi$ &  CPU Time (mean, in secs.) & CPU Time (median, in secs.)& \# IPM Iterations (mean) & \# Relaxations (mean)\\
\midrule
\multirow{3}{*}{DSR} & CVaR & 0.730 & 0.100 & 177.100 & 6.500 \\
 & EXP & 0.770 & 0.080 & 188.600 & 6.500 \\
 & SUM & 0.610 & 0.090 & 152.300 & 6.600 \\\midrule
\multirow{3}{*}{ISR} & CVaR & 0.800 & 0.120 & 190.500 & 6.400 \\
& EXP & 1.060 & 0.140 & 240.300 & 6.500 \\
 & SUM & 0.920 & 0.100 & 198.100 & 6.400 \\\midrule
\multirow{3}{*}{UNC} & CVaR & 0.300 & 0.060 & 134.900 & 6.600 \\
 & EXP & 0.420 & 0.050 & 156.000 & 6.600 \\
 & SUM & 0.280 & 0.050 & 125.800 & 6.500 \\
\bottomrule
\end{tabular}
\caption{Numerical behavior of the single-edge formulations. The table reports CPU times, interior-point method iterations, and branch-and-bound statistics.}
\label{tab:single_edge_numerical_strength}
\end{table}

Finally, Table~\ref{tab:single_edge_numerical_strength} reveals also that the computational impact of the objective function remains relatively limited. As expected, the exponential penalty produces the largest computational burden due to the additional exponential cone structures, whereas the sum objective is consistently the easiest to solve. Nevertheless, even under the most challenging ISR--EXP combination, the average running time remains close to one second, confirming the practical tractability of the proposed conic framework in the single-edge setting.

Overall, these results demonstrate that the proposed conic reformulations effectively isolate the nonlinear queueing interactions into tractable low-dimensional conic blocks, thereby preserving excellent numerical behavior throughout the entire benchmark set.

\subsubsection*{Managerial Insights}

The experiments reveal that congestion-aware modeling plays a critical role in accurately assessing tail performance. As the origin-delay parameter $\kappa_2$ increases, cache misses become substantially more expensive and the degradation in service quality becomes increasingly pronounced under the congestion-aware regimes. In contrast, the uncongested approximation systematically underestimates the response times experienced under heavy traffic conditions. Moreover, the ISR regime consistently produces larger tail delays than DSR, reflecting the stronger congestion interactions induced by the shared queueing structure. These results indicate that ignoring congestion effects may lead operators to underestimate the infrastructure required to maintain acceptable reliability levels.

The experiments also show that both the optimization criterion and the size of the origin layer strongly influence the robustness of the resulting CDN designs. The sum objective tends to prioritize aggregate efficiency at the expense of poorly served users, thereby producing the largest tail response times. Conversely, the CVaR objective provides the strongest protection against extreme delays, while the exponential objective yields an intermediate compromise between efficiency and robustness. Regarding the origin infrastructure, increasing the number of origin servers improves service quality by reducing the distance associated with cache misses and providing greater routing flexibility. However, the improvements exhibit clear diminishing returns, suggesting that beyond a moderate number of origins, further infrastructure expansion yields only limited additional reductions in tail latency. Overall, the results highlight the importance of jointly accounting for congestion effects, robustness-oriented objectives, and origin-layer sizing when designing reliable CDN infrastructures.

\subsection{Multiple-Edge CDN}

We next evaluate the whole mixed-integer conic formulation \eqref{eq:cnd} in the general multi-edge setting, where the models jointly determine edge-server locations, demand assignments, miss-routing decisions, and service capacities. In contrast with the single-edge regime, this setting introduces a substantially richer combinatorial structure due to the interaction between binary allocation variables and nonlinear congestion effects. Consequently, these experiments evaluate not only the numerical behavior of the conic queueing liftings, but also their interaction with the discrete location-allocation layer.

\subsubsection*{Computational Performance}

Table~\ref{tab:multi_cpu_combined} reports the average computational results obtained with the proposed models under the different service regimes and objective functions.

\begin{table}[ht]
\centering
\begin{tabular}{p{0.7cm}cc|rrr|rrr|rrr}
\toprule
 && & \multicolumn{3}{c}{CPU Time (secs.)} &
 \multicolumn{3}{c}{Unsolved (\%)} &
 \multicolumn{3}{c}{MIP Gap (\%)}\\
\midrule
Regime & $|\mathcal I|$ & $|\J|$
& SUM & CVaR & EXP
& SUM & CVaR & EXP
& SUM & CVaR & EXP \\
\midrule
\multirow{10}{*}{DSR} & \multirow{2}{*}{10} & 2 & 0.75 & 0.57 & 0.81 & 0 & 0 & 0 & 0 & 0 & 0 \\
 &  & 3 & 13.72 & 3.48 & 17.41 & 0 & 0 & 0 & 0 & 0 & 0 \\
 & \multirow{2}{*}{20} & 2 & 507.92 & 72 & 301.71 & 0 & 0 & 0 & 0 & 0 & 0 \\
 &  & 3 & 5821.54 & 538.42 & 5526.21 & 62.07 & 0 & 61.40 & 25.76 & 0 & 5.68 \\
 & \multirow{2}{*}{30} & 2 & 5636.98 & 256.59 & 3843.21 & 52.33 & 0 & 30.23 & 18.55 & 0 & 2.87 \\
 &  & 3 &7200 & 2053.74 &7200 & 100 & 7.69 & 100 & 77.42 & 2.80 & 15.53 \\
 & \multirow{2}{*}{40} & 2 &7200 & 586.31 & 7000.10 & 100 & 0 & 89.80 & 47.64 & 0 & 9.59 \\
 &  & 3 &7200 & 3403.77 &7200 & 100 & 9.09 & 100 & 92.71 & 6.16 & 16.77 \\
 & \multirow{2}{*}{50} & 2 &7200 & 1605.54 &7200 & 100 & 0 & 100 & 63.20 & 0 & 16.83 \\
 &  & 3 &7200 & 6985.69 &7200 & 100 & 80 & 100 & 95.36 & 36.86 & 16.94 \\\midrule
\multirow{10}{*}{ISR} & \multirow{2}{*}{10} & 2 & 1.43 & 0.65 & 1.26 & 0 & 0 & 0 & 0 & 0 & 0 \\
 &  & 3 & 209.60 & 12.15 & 71.88 & 0 & 0 & 0 & 0 & 0 & 0 \\
 & \multirow{2}{*}{20} & 2 & 925.56 & 57.19 & 525.41 & 3.33 & 0 & 1.11 & 0.62 & 0 & 0.13 \\
 &  & 3 & 6149.68 & 1218.77 & 6084.54 & 72.41 & 10.34 & 68.42 & 40.88 & 0.73 & 10.08 \\
 & \multirow{2}{*}{30} & 2 & 5595.86 & 177.77 & 4083.35 & 60.47 & 0 & 33.72 & 19.50 & 0 & 3.00 \\
 &  & 3 &7200 & 1455.16 &7200 & 100 & 7.69 & 100 & 71.46 & 0.35 & 11.38 \\
 & \multirow{2}{*}{40} & 2 &7200 & 715.48 & 7114.11 & 100 & 2.04 & 95.83 & 49.94 & 0.27 & 10.49 \\
 &  & 3 &7200 & 5815.30 &7200 & 100 & 54.55 & 100 & 79.59 & 15.58 & 12.36 \\
 & \multirow{2}{*}{50} & 2 &7200 & 1235.17 &7200 & 100 & 4.65 & 100 & 62.20 & 0.42 & 15.27 \\
 &  & 3 &7200 & 6668.13 &7200 & 100 & 60 & 100 & 85.09 & 20.48 & 12.94 \\\midrule
\multirow{10}{*}{UNC} & \multirow{2}{*}{10} & 2 & 0.17 & 0.11 & 0.18 & 0 & 0 & 0 & 0 & 0 & 0 \\
 &  & 3 & 2.74 & 0.53 & 2.44 & 0 & 0 & 0 & 0 & 0 & 0 \\
 & \multirow{2}{*}{20} & 2 & 48.78 & 3.98 & 34.25 & 0 & 0 & 0 & 0 & 0 & 0 \\
 &  & 3 & 2477.65 & 27.70 & 1941.71 & 25.86 & 0 & 15.79 & 6.51 & 0 & 1.14 \\
 & \multirow{2}{*}{30} & 2 & 1771.63 & 11.13 & 1177.62 & 12.79 & 0 & 8.14 & 2 & 0 & 0.46 \\
 &  & 3 &7200 & 36.07 & 5711.69 & 100 & 0 & 53.85 & 39.49 & 0 & 3.94 \\
 & \multirow{2}{*}{40} & 2 & 5725.72 & 27.73 & 4218.94 & 65.31 & 0 & 41.67 & 19.01 & 0 & 3.62 \\
 &  & 3 &7200 & 978.92 &7200 & 100 & 0 & 100 & 60.82 & 0 & 6.97 \\
 & \multirow{2}{*}{50} & 2 &7200 & 57.82 & 6426.14 & 100 & 0 & 71.43 & 41 & 0 & 7.79 \\
 &  & 3 &7200 & 832.98 &7200 & 100 & 0 & 100 & 67.08 & 0 & 7.45 \\
\bottomrule
\end{tabular}
\caption{Computational performance of the multi-edge CDN instances.}
\label{tab:multi_cpu_combined}
\end{table}

From these results, we derive three main conclusions. First, the computational difficulty is driven primarily by the joint increase in the number of demand points and the number of deployed edge servers. Small instances with $|\mathcal I|=10$ are solved almost instantaneously across all regimes and objectives, even for $|\J|=3$. However, the transition to medium-scale instances already produces a dramatic increase in complexity. In particular, moving from $|\J|=2$ to $|\J|=3$ significantly enlarges the assignment and congestion interactions, leading to severe scalability issues for the largest instances. For example, under the DSR regime with the SUM objective, the average CPU time increases from $507.92$ seconds for $(|\mathcal I|,|\J|)=(20,2)$ to the time limit for $(20,3)$, where more than $62\%$ of the instances remain unsolved. Similar saturation effects appear systematically for $|\mathcal I|\geq 30$, especially when $|\J|=3$, where many formulations reach the time limit with substantial optimality gaps.

Second, the congestion-aware formulations remain computationally tractable despite the additional queueing-based nonlinearities introduced by the proposed conic models. As expected, the uncongested approximation (UNC) is systematically the easiest regime to solve, since it avoids the stochastic congestion interactions. Nevertheless, the additional computational burden induced by the queueing formulations remains moderate for small and medium-size instances. Moreover, the DSR regime consistently outperforms ISR in terms of computational behavior. This effect becomes particularly visible for difficult configurations. For instance, under the SUM objective with $(|\mathcal I|,|\J|)=(20,3)$, DSR reports an average optimality gap of $25.76\%$, whereas ISR increases to $40.88\%$. Similarly, for $(|\mathcal I|,|\J|)=(40,3)$ under CVaR, DSR leaves only $9.09\%$ of the instances unsolved with an average gap of $6.16\%$, while ISR leaves $54.55\%$ unsolved and reaches a gap of $15.58\%$. This behavior reflects the stronger nonlinear coupling induced by the integrated $\mathrm{M/H}_{|\Theta|}/1$ representation of ISR, where cache hits and misses interact through a shared congestion process.

Third, the optimization criterion has a major impact on computational tractability. The CVaR objective is by far the easiest formulation to solve across all regimes and instance sizes. Remarkably, most CVaR instances are solved to optimality or near optimality, even for relatively large configurations. For example, under UNC all CVaR instances up to $(|\mathcal I|,|\J|)=(50,3)$ are solved without any unsolved instances and with zero optimality gap. Even under the more difficult DSR and ISR regimes, the CVaR formulations remain highly tractable, with significantly smaller CPU times and optimality gaps than the SUM and EXP objectives. In contrast, the SUM and EXP objectives become considerably more difficult as the system size grows. The SUM objective is particularly challenging because minimizing the average response time strongly couples assignment decisions with congestion effects, whereas the EXP objective additionally introduces exponential cone structures that further increase the numerical burden. Indeed, for large instances with $|\J|=3$, both objectives frequently hit the time limit with unresolved gaps above $70\%$ under the congestion-aware regimes. Nevertheless, even in these difficult cases, the proposed mixed-integer conic reformulations still produce feasible high-quality solutions with moderate average gaps for many medium-scale instances, confirming the practical viability of the proposed framework.

\subsubsection*{Managerial Insights}

From a service-quality perspective, Table~\ref{tab:multi_quality} reports the deviations of the expected response times obtained under the congestion-aware formulations with respect to the corresponding uncongested (UNC) approximation. More precisely, the table summarizes the relative deterioration in the expected response times induced by explicitly incorporating queueing congestion effects into the CDN design process. The indicators include the mean, median, and upper-tail deviations (percentiles $90$ and $95$), thereby providing a global view of both the average performance loss and the worst-case congestion amplification effects.

\begin{table}[h]
\begin{tabular}{llrrrrrr}
\toprule
$\Phi$ & Regime & Mean dev. & Median dev. & P$_{90}$ dev. & P$_{95}$ dev. \\
\midrule
\multirow{2}{*}{SUM} & DSR & 5.33 & 0.82 & 7.00 & 7.47 \\
 & ISR & 25.85 & 9.72 & 13.38 & 12.30 \\\midrule
\multirow{2}{*}{CVaR} & DSR & 2.06 & 0.60 & 1.27 & 1.29 \\
 & ISR & 13.18 & 6.41 & 11.63 & 11.43 \\\midrule
\multirow{2}{*}{EXP} & DSR & 5.03 & 0.75 & 6.91 & 6.83 \\
 & ISR & 26.31 & 10.11 & 14.07 & 12.73 \\
\bottomrule
\end{tabular}
\caption{Deviation statistics of the response times under the congestion-aware regimes.\label{tab:multi_quality}}
\end{table}

The results reveal a pronounced structural difference between the DSR and ISR congestion regimes. The DSR regime consistently produces comparatively small deviations across all objective functions, with mean deviations ranging between approximately $2\%$ and $5\%$. Moreover, the median deviations remain below $1\%$ in all cases, and even the upper-tail indicators remain moderate, with P$_{90}$ and P$_{95}$ deviations around $7\%$ for the SUM and EXP objectives and close to $1\%$ for CVaR. These results indicate that the disaggregated queueing structure generates relatively stable response-time distributions and limits congestion propagation effects across the network.

In contrast, the ISR regime exhibits substantially larger deviations under all optimization criteria. The SUM and EXP objectives produce mean deviations above $25\%$, together with median deviations around $10\%$ and upper-tail deviations exceeding $12\%$--$14\%$. This behavior reflects the stronger congestion interactions induced by the integrated $\mathrm{M/H}_{|\Theta|}/1$ representation, where cache hits and misses compete within a shared queueing system. Consequently, congestion effects become significantly amplified, especially under heavily loaded configurations, leading to markedly different response-time predictions and infrastructure configurations relative to the uncongested approximation.

The table also highlights the strong impact of the optimization criterion itself. Under both DSR and ISR, the CVaR objective systematically produces the smallest deviations, both on average and in the upper tails of the response-time distribution. In particular, under ISR, the mean deviation decreases from approximately $26\%$ under SUM and EXP to only $13.18\%$ under CVaR, while the tail deviations are also substantially reduced. This behavior confirms that risk-aware optimization criteria are considerably more robust to congestion amplification effects, especially when queue interactions are strongly coupled. Overall, these results reinforce the importance of simultaneously incorporating congestion-aware queueing models and tail-oriented optimization criteria when designing large-scale CDN infrastructures subject to stochastic service dynamics.

It is important to note that these differences are also influenced by the common budget restriction imposed on the service capacities. In the experiments, the upper bound $\Gamma$ of the budget constraint is calibrated so as to guarantee feasibility simultaneously for both the DSR and ISR formulations according to Proposition \ref{prop:budget_feasibility}. Since the ISR regime is structurally more demanding in terms of congestion management, the same budget bound becomes effectively tighter under ISR, limiting the possibility of increasing service capacities to mitigate congestion effects. As a consequence, ISR solutions operate under significantly higher congestion levels. In contrast, the DSR regime benefits from greater flexibility in distributing the available capacities across the disaggregated service subsystems, allowing the model to adapt the capacities more efficiently and therefore produce substantially lower congestion levels and smaller response-time deviations.

Table~\ref{tab:objective_cross_deviation} reports the mean relative deviation obtained when solutions optimized under a given objective function are evaluated using the remaining criteria. The results reveal that the SUM and EXP objectives produce very similar solutions, as evidenced by the negligible deviations observed when exchanging their evaluations: the SUM solution evaluated under the EXP criterion yields a deviation of only $0.17\%$, while the EXP solution evaluated under the SUM objective exhibits a deviation of $0.44\%$. This indicates that both criteria promote closely related infrastructure configurations and congestion patterns. In contrast, the CVaR objective generates substantially different solutions. In particular, the CVaR-minimized solution incurs a deviation of $72.22\%$ when evaluated under the SUM objective and $24.12\%$ under the EXP criterion, reflecting the conservative nature of the CVaR formulation, which prioritizes the mitigation of extreme response times over average efficiency. Conversely, the SUM and EXP solutions exhibit more moderate degradations when assessed under the CVaR criterion, namely $16.91\%$ and $13.59\%$, respectively. Overall, these results highlight a clear trade-off between aggregate efficiency and tail-risk protection, with SUM and EXP emphasizing average performance, whereas CVaR produces more robust but structurally different CDN configurations.
\begin{table}
\begin{tabular}{lrrr}
\toprule
$\Phi$ & SUM & CVAR & EXP \\
\midrule
SUM & 0.00 & 16.91 & 0.17 \\
CVAR & 72.22 & 0.00 & 24.12 \\
EXP & 0.44 & 13.59 & 0.00 \\
\bottomrule
\end{tabular}
\caption{Mean relative deviation of each solution evaluated under the three objective functions.}
\label{tab:objective_cross_deviation}
\end{table}
Finally, Table~\ref{tab:p_effect_response_times} summarizes the impact of increasing the number of edge servers with respect to the single-server configuration ($|\J|=1$) on the response times. The reported values represent the average percentage improvement obtained when the same instance is solved with additional edge servers while keeping all remaining parameters fixed. Overall, the results show that enlarging the edge infrastructure systematically improves the response-time performance across all objective functions and queueing regimes. As expected, increasing the number of edge servers provides additional flexibility for distributing demand across the network, thereby reducing both congestion levels and propagation delays experienced by users.

Under the DSR regime, the improvements are particularly large and highly stable across all objectives. Increasing the number of edge servers from $|\J|=1$ to $|\J|=2$ already yields average improvements around $44\%$--$46\%$, while configurations with $|\J|=3$ achieve improvements exceeding $61\%$ in all cases. A very similar behavior is observed under the uncongested UNC approximation, where the improvements are slightly larger and reach values above $62\%$ for the CVaR and EXP objectives. These results indicate that, under DSR and UNC, the additional edge infrastructure can be exploited very efficiently to redistribute the traffic load and reduce the response times throughout the network.

The ISR regime also benefits substantially from increasing the number of edge servers, although the magnitude of the improvements is systematically smaller than under DSR and UNC. For example, under the SUM objective, increasing the number of servers to $|\J|=3$ improves the response times by approximately $47.57\%$, compared with more than $61\%$ under DSR and UNC. Similar differences appear for the EXP objective, where ISR achieves an improvement of $48.43\%$, whereas DSR and UNC exceed $61\%$. Even under the CVaR objective, which remains the most stable criterion across all regimes, ISR still produces noticeably smaller gains than the alternative formulations.

These results suggest that the integrated queueing dynamics inherent to ISR reduce the marginal benefit obtained from deploying additional edge servers. Since the ISR regime models a stronger coupling between cache-hit and cache-miss traffic through a shared queueing process, congestion effects propagate more intensely across the system and limit the efficiency gains that can be achieved through traffic redistribution. In contrast, the DSR formulation decouples the congestion dynamics into independent service subsystems, allowing the additional infrastructure to absorb the demand more effectively and leading to substantially larger scalability gains. Consequently, while increasing the number of edge servers systematically improves service quality in all regimes, the benefits are significantly stronger under DSR and UNC than under ISR.
\begin{table}[ht]
\centering
\renewcommand{\arraystretch}{0.95}

\begin{tabular}{llcc}
\toprule
$\Phi$ & Regime
& $|\J|=2$ & $|\J|=3$ \\
\midrule
\multirow{3}{*}{SUM} & DSR & 45.74 & 61.10 \\
& ISR & 40.47 & 47.57 \\
 & UNC & 46.21 & 61.63 \\\midrule
\multirow{3}{*}{CVaR} & DSR & 44.17 & 61.68 \\
 & ISR & 41.73 & 55.60 \\
 & UNC & 44.24 & 62.57 \\\midrule
\multirow{3}{*}{EXP} & DSR & 45.98 & 61.83 \\
 & ISR & 40.74 & 48.43 \\
 & UNC & 46.48 & 62.34 \\
\bottomrule
\end{tabular}
\caption{Effect of increasing the number of edge servers on response times, measured as the percentage improvement with respect to the same instance with $|\J|=1$.}
\label{tab:p_effect_response_times}
\end{table}
Taken together, the computational experiments support two main conclusions. First, the proposed conic liftings lead to formulations that are computationally reliable for single-edge and moderately sized multi-edge CDN design problems, with the main source of difficulty arising from the combinatorial assignment layer rather than from the conic representation of queueing effects. Second, the congestion-aware models produce materially different service-quality assessments and design trade-offs from the uncongested approximation, especially under shared queueing dynamics and tail-sensitive objectives.

\section{Conclusion and Further Research}\label{sec:conclusion}

This paper develops a mixed-integer conic optimization framework for the design of content delivery networks (CDNs) under congestion effects. The proposed approach integrates facility location decisions, demand allocation, queueing-based performance modeling, and service capacity planning within a unified optimization framework. Unlike classical CDN design models based on deterministic distance approximations or clustering heuristics, the framework explicitly incorporates stochastic congestion phenomena arising from cache-hit and cache-miss dynamics at edge servers.

A key methodological contribution lies in the treatment of nonlinear queueing expressions through conic lifting techniques rather than approximation or linearization. It is shown that response-time expressions derived from the underlying queueing models can be decomposed into convex primitive components such as reciprocal terms, affine ratios, and quadratic-over-linear structures, all of which admit exact conic representations. This result highlights that conic optimization is not merely a computational tool, but the natural modeling language for congestion-aware network design problems.

Two service regimes are analyzed. In the disaggregated regime, cache hits and cache misses are processed through independent queueing subsystems, capturing heterogeneous service dynamics in a structured way. In the integrated regime, both request classes interact within a shared congestion system modeled as an {\rm M/H/1} queue. These complementary formulations provide alternative perspectives on how caching mechanisms interact with congestion effects in distributed service architectures.

From an optimization perspective, the resulting models combine binary assignment decisions, continuous facility location variables, nonlinear queueing constraints, and conic representations within exact mixed-integer conic formulations. The paper derives explicit conic liftings for all nonlinear performance expressions and shows that the resulting models can be efficiently handled by modern mixed-integer conic optimization solvers. In this sense, the proposed framework unifies location science, queueing theory, and conic optimization within a single modeling paradigm.

The framework is also flexible in terms of service-quality objectives. It accommodates multiple criteria, including total response time minimization, worst-case performance, CVaR-based risk measures, ordered weighted aggregation, norm-based objectives, and exponential penalties. This versatility allows the model to capture both efficiency-oriented and risk-averse design perspectives, making it suitable for a wide range of CDN planning scenarios.

Finally, several directions for future research are identified. First, large-scale computational studies using realistic Internet topologies and traffic traces are needed to evaluate scalability and practical performance. Second, extending the framework to dynamic and time-dependent environments represents a natural next step, including time-varying demand, evolving content popularity, adaptive caching strategies, and stochastic capacity fluctuations. Such extensions may lead to multi-period stochastic or robust optimization models.

Another important direction is the integration of cache placement decisions into the optimization framework. In the current formulation, cache-hit probabilities are exogenous, although in practice they depend on caching policies, storage capacities, and replication strategies. Joint optimization of caching, replication, and network design would significantly enhance the realism and applicability of the model.

From a methodological standpoint, future work may focus on stronger conic reformulations, tighter relaxations, and decomposition methods such as Benders decomposition, branch-and-cut, and perspective reformulations. Developing scalable algorithms tailored to the specific queueing-induced conic structure remains an important research challenge. Finally, although the paper focuses on CDNs, the proposed framework extends naturally to other distributed service systems such as edge computing, cloud-fog architectures, telecommunications networks, transportation systems, and latency-sensitive AI inference platforms.

\vspace*{0.4cm}

\noindent{\bf Funding: } 
This research has been financially supported by grants PID2020-114594GB-C21 and PID2024-156594NB-C21 funded by MICIU/AEI/10.13039/501100011033; FEDER + Junta de Andaluc\'ia project C‐EXP‐139‐UGR23; the IMAG--Mar\'ia de Maeztu grant CEX2020-001105-M/AEI/10.13039/501100011033; and the IMUS--Mar\'ia de Maeztu grant CEX2024-001517-M/AEI/10.13039/501100011033.



\appendix

\section{Products by Binary Decisions: McCormick-Type Envelopes}\label{app_McCormick}

The functions arising in the models considered here depend jointly on continuous variables and binary allocation decisions, and their structure involves compositions of affine mappings, reciprocal transformations, and weighted sums. A central observation is that these functions can be decomposed into simpler components that admit conic representations. By exploiting closure properties of conic representability under affine composition, addition, and perspective transformations, one can construct extended formulations for the full expressions in a systematic way.

A fundamental ingredient in this construction is the convexification of products involving binary variables. Given a binary variable $z \in \{0,1\}$, a function $f:\mathcal{V} \subseteq \mathbb{R}^n \to \mathbb{R}$ with finite bounds $LB \le f(v) \le UB, \forall v\in\mathcal{V}$, and the expression $w = z f(v)$, the convex hull of the graph of this product is described by the inequalities
$$
LB\, z \le w \le UB\, z,
$$
$$
f(v) - UB (1-z) \le w \le f(v) - LB (1-z).
$$
These inequalities correspond to the McCormick envelopes~\citep{McCormick1976} for the case in which one factor is binary, and they yield an exact representation of the product. Indeed, if $z=0$ the system enforces $w=0$, while if $z=1$ it implies $w=f(v)$.

Equivalently, this formulation describes the convex hull of the disjunction
$$
(z,v,w) \in 
(
\{0\} \times [LB,UB] \times \{0\})\cup
(\{1\} \times [LB,UB] \times \{f(v)\}),
$$
and therefore introduces no relaxation gap when valid bounds are available. In the particular case $LB=0$, which frequently arises in our setting since many quantities are nonnegative, the envelopes simplify to
\begin{align*}
0 &\le w \le UB\, z,\\
f(v) - UB (1-z) &\le w \le f(v).
\end{align*}
When these constraints are embedded in problems whose objective drives $f(v)$, and therefore $w$, to decrease (analogously, to increase) towards its smallest (largest) feasible value, then some of the upper (lower) bounding inequalities may become redundant. In particular, constraints of the form $w \le UB z$ or $w \le f(v)$ can often be omitted without affecting optimality. This observation can be exploited to slightly reduce the size of the resulting formulations. A particularly common case arises when $f$ is the identity function, in which case the above system provides an exact linearization of the bilinear relation $w = z v$.

The importance of these constructions lies in their ability to couple binary allocation decisions with conic constraints while preserving convexity. As will be shown in the subsequent sections, the nonlinear performance expressions that arise in the CDN design problem can be fully expressed using functions that admit such representations. This yields a unified and tractable modeling framework in which conic lifting is not auxiliary, but intrinsic to the formulation.

\section{Instance Generation from Internet Topology Data}\label{app:instances}

The computational study is based on instances derived from the CAIDA MIDAR Internet topology dataset~\citep{ITDK}. The construction procedure follows a structured preprocessing pipeline designed to preserve the main structural properties of the original Internet graph while generating instances of controlled size.

The generation process uses three components of the CAIDA dataset: geographical node information, router-level connectivity data, and autonomous-system identifiers. The overall procedure consists of four stages.

First, a subset of $20{,}000$ nodes is extracted from the geographical database using reservoir sampling directly from the compressed files. This avoids loading the full dataset into memory while preserving a uniform random sample. For each selected node, the latitude and longitude coordinates define its spatial position $a_i \in \mathbb{R}^2$.

Second, the degree of each sampled node is computed using the complete connectivity graph. Importantly, the degree is measured with respect to the original CAIDA topology rather than the sampled subgraph. Thus, if ${\rm deg}_i$ denotes the degree of node $i$, this quantity reflects its global structural importance within the Internet infrastructure.

Third, the origin servers are selected deterministically as the $8$ sampled nodes with largest degree. This choice reflects the interpretation of highly connected routers as Internet backbone hubs naturally associated with origin infrastructure.

Finally, the demand points are generated by sampling $|\I|$ nodes uniformly at random from the remaining nodes after removing the selected origins. Each demand point inherits both its geographical coordinates and degree information from the underlying topology.

The request intensities are designed to combine structural heterogeneity with stochastic variability. Let $\widetilde{\deg}_i \in [0,1]$ denote the normalized degree of node $i$. Independently, Gamma random variables $\gamma_i \sim \Gamma(2,1)$ are generated and normalized into $\widetilde{\gamma}_i \in [0,1]$. The request intensity associated with demand point $i$ is then defined as
$$
\xi_i
=
0.7(0.2+\widetilde{\deg}_i)
+
0.3(0.2+\widetilde{\gamma}_i),
$$
followed by the normalization
$$
\frac{1}{|\I|}
\sum_{i\in\I}\xi_i
=
1.
$$

This construction yields heterogeneous traffic patterns positively correlated with node connectivity while avoiding purely deterministic structures. The additive constant prevents nearly inactive demand points, whereas the convex combination balances topology-driven and stochastic variability.

The cache-hit probabilities are generated as functions of the node degree. Specifically, for each demand point,
$$
p_i
=
0.55
+
(0.95-0.55)
\frac{\deg_i-\deg_{\min}}
{\deg_{\max}-\deg_{\min}},
$$
where $\deg_{\min}$ and $\deg_{\max}$ denote the minimum and maximum degree values among the selected demand nodes. Consequently, highly connected regions exhibit larger cache-hit probabilities, reflecting the empirical observation that highly active regions tend to benefit from improved caching efficiency.

All random operations are controlled through fixed seeds, and five independent replications are generated for each parameter configuration.
\end{document}

%% file: 00_macros.tex
\usepackage{geometry}

\usepackage{amsmath, amssymb, amsfonts, amsthm}
\allowdisplaybreaks
\usepackage{mathtools}
\usepackage{bm}

\usepackage{graphicx}
\usepackage{subcaption}
\usepackage{tikz}
\usetikzlibrary{arrows.meta, positioning, calc, fit,
                shapes.geometric, decorations.pathreplacing}
\usepackage{arydshln}
\usepackage{dashrule}
\usepackage{empheq}

\usepackage{booktabs}
\usepackage{multirow}
\usepackage{longtable}
\usepackage{makecell}

\usepackage{algorithm}
\usepackage{algpseudocode}

\usepackage{xcolor}
\usepackage{tcolorbox}
\tcbuselibrary{breakable}

\usepackage{natbib}

\usepackage{adjustbox}
\usepackage{placeins}
\usepackage{enumitem}
\usepackage{hyperref}

\hypersetup{
  colorlinks = true,
  linkcolor  = blue,
  citecolor  = blue,
  urlcolor   = blue
}

\newtheorem{theorem}{Theorem}
\newtheorem{proposition}{Proposition}
\newtheorem{lemma}{Lemma}
\newtheorem{corollary}{Corollary}
\newtheorem{remark}{Remark}

\DeclareMathOperator*{\minimize}{minimize}
\DeclareMathOperator{\st}{subject\; to}

\newcommand{\dsum}{\displaystyle\sum}

\newcommand{\R}{\mathbb{R}}

\newcommand{\E}{\mathbb{E}}

\newcommand{\LL}{\mathcal{L}}

\renewcommand{\SS}{\mathcal{S}}

\newcommand{\K}{\mathcal{K}}
\newcommand{\I}{\mathcal{I}}
\newcommand{\J}{\mathcal{J}}

\newcommand{\D}{\mathrm{D}}

\newcommand{\T}{\mathrm{T}}

\definecolor{terracota}{RGB}{130,36,31}
\definecolor{punkpink}{RGB}{255,20,147}
\newcommand{\highlight}[1]{{\color{terracota} #1}}

\algnewcommand\algorithmicinput{\textbf{Input:}}
\algnewcommand\algorithmicoutput{\textbf{Output:}}
\algnewcommand\Input{\item[\algorithmicinput]}
\algnewcommand\Output{\item[\algorithmicoutput]}

\let\origmaketitle\maketitle
\def\maketitle{
	\begingroup
	\def\uppercasenonmath##1{} 
	\let\MakeUppercase\relax 
	\origmaketitle
	\endgroup
}

\usepackage{graphicx}
\usepackage{booktabs}
\usepackage{subcaption}
\usepackage{multirow}

\usepackage{fontawesome5}